\documentclass{amsart}

\usepackage{amssymb,amsmath,esint}
\usepackage[all]{xy}
\usepackage{graphicx}
\usepackage{enumerate}

\usepackage{colortbl}
\usepackage{cite}
\usepackage{xcolor}

\usepackage{amsrefs}
\usepackage{soul}

\usepackage{verbatim}

\usepackage{geometry}
\usepackage{comment}
\usepackage{graphicx,caption,subcaption}
\usepackage{tikz}
\usepackage{multirow} 

\colorlet{linkequation}{blue}
\usepackage[colorlinks,plainpages=true,pdfpagelabels,hypertexnames=true,colorlinks=true,pdfstartview=FitV,linkcolor=blue,citecolor=red,urlcolor=black]{hyperref}
\usepackage{hyperref}

\PassOptionsToPackage{unicode}{hyperref}
\PassOptionsToPackage{naturalnames}{hyperref}
\definecolor{dgreen}{rgb}{0,0.5,0}
\definecolor{violet}{rgb}{0.5,0,0.5}
\definecolor{dred}{rgb}{0.7,0,0}
\definecolor{ddred}{rgb}{0.5,0,0}
\definecolor{dblue}{rgb}{0,0,0.5}
\definecolor{ddblue}{rgb}{0,0,0.3}
\definecolor{llgray}{rgb}{0.9,0.9,0.9}
\definecolor{lgray}{rgb}{0.7,0.7,0.7}

\newtheorem{defn}{Definition}[section]
\newtheorem{lemma}[defn]{Lemma}
\newtheorem{proposition}[defn]{Proposition}
\newtheorem{theorem}[defn]{Theorem}
\newtheorem{cor}[defn]{Corollary}

\numberwithin{equation}{section}

\newcommand{\bq}{\begin{equation}}
\newcommand{\eq}{\end{equation}}

\newcommand{\R}{{ \mathbb{R}  }}

\newcommand{\bke}[1]{\left( #1 \right)}
\newcommand{\bkt}[1]{\left[ #1 \right]}
\newcommand{\bket}[1]{\left\{ #1 \right\}}
\newcommand{\norm}[1]{\left\Vert #1 \right\Vert}
\newcommand{\abs}[1]{\left| #1 \right|}

\newenvironment{pfthm1}{{\par\noindent\bf
           Proof of Theorem~\ref{THM1}. }}{\hfill\fbox{}\par\vspace{.2cm}}
\newenvironment{pfthm3}{{\par\noindent\bf
           Proof of Theorem~\ref{THM3}. }}{\hfill\fbox{}\par\vspace{.2cm}}

\DeclareMathOperator{\supp}{supp}

\AtBeginDocument{%
   \def\MR#1{}
}
\nocite{*}

\begin{document}

%\title{ Keller--Segel--Navier--Stokes systems with signal--dependent power--law decay in general sensitivities }
%\title{ Keller--Segel systems with/without fluid coupling with signal--dependent power--law decay in general sensitivities }
%\title{Keller–-Segel Systems With/Without Fluid Coupling involving General Sensitivities with Signal-Dependent Power-Law Decay}
\title[Keller--Segel--Navier--Stokes systems]{Keller--Segel--Navier--Stokes systems involving general sensitivities with Signal-Dependent Power-Law Decay}

\author[J. Ahn]{Jaewook Ahn}
\address{J. Ahn: Department of Mathematics, Dongguk University, Seoul 04620, Republic of Korea}
\email{jaewookahn@dgu.ac.kr}

\author[S. Hwang]{Sukjung Hwang}
\address{S. Hwang: Department of Mathematics Education, Chungbuk National University, Cheongju 28644, Republic of Korea}
\email{sukjungh@chnu.ac.kr}

\thanks{J. Ahn's work  is partially supported by RS-2024-00336346 and RS-2025-24523482. S. Hwang's work is partially supported by NRF-2022R1F1A1073199 and RS-2025-23324098.}

\date{}

%%%%%%%%%%%%%%%%%%%
\makeatletter
\@namedef{subjclassname@2020}{%
  \textup{2020} Mathematics Subject Classification}
\makeatother
\subjclass[2020]{35B45, 35A09, 35K57, 35Q35, 35Q92}
%{2020 AMS Subject Classification}\, :\,  35A01, 35K55, 35Q84, 92B05

%{Keywords}\,:\, Porous medium equation, weak solution, Wasserstein space, a bounded domain 
\keywords{Keller--Segel--Navier--Stokes system, uniform boundedness, asymptotics}

%\Addresses

%%%%%%%%%%%%%%%%%%%%%%%%%%%%%%%%%%%%%%%%%%%%%%%%%%%%%%%%%

\begin{abstract}
This paper investigates a two-dimensional Keller--Segel--Navier--Stokes system with a tensor-valued chemotactic sensitivity $S(x,n,c)$. Under a signal-dependent power-decay condition $|S(x,n,c)| \le s_0 (s_1+c)^{-\gamma}$, we establish the global existence and uniform-in-time boundedness of classical solutions for both fluid-coupled ($\gamma > 1/2$) and fluid-free ($\gamma > 0$) systems. The proof relies on a sequence of localized energy estimates, including the $L^{2}_{\rm loc}$-smallness of the weighted gradient of the signal concentration, to overcome the mathematical difficulties arising from signal production and fluid transport. Furthermore, under specific structural assumptions on the sensitivity tensor, we prove that solutions of the fluid-free system converge exponentially to the spatially homogeneous steady state. To this end, we establish an interpolation inequality involving the H\"older norm, which is of independent interest and seems to have broad applications.
\end{abstract}
\maketitle

%%%%%%%%%%%%%%%%%%%%%%%%%%%%%%%%%%%%%%%%%%%%%%%%%%%%%%%%%
%{
%  \hypersetup{linkcolor=black}
%  \tableofcontents
%}

%\hypersetup{linkcolor=black}
%\tableofcontents

%%%%%%%%%%%%%%%%%%%%%%%%%%%%%%%%%%%%%%%%%%%%%%%%%%%%%%%%%%%%

\section{Introduction}

The question of blow-up versus boundedness in Keller--Segel type systems is strongly influenced by the structural properties of the chemotactic sensitivity. 
Since the pioneering work of Keller and Segel~\cite{Keller_Segel_1970}, the classical model
\[
n_{t} = \Delta n -   \nabla \cdot (n\chi(c)\nabla c), \qquad \tau c_{t}=\Delta c + n -c
\]
with a constant sensitivity 
$\chi(c)=\chi_0>0$
 has been investigated extensively. A notable feature of this system is the possibility of finite-time blow-up. For instance, it may occur if the initial mass exceeds a critical mass threshold in two dimensions~\cite{Jager_Luckhaus_1992,Nagai_Senba_Yoshida_1997,Senba_Suzuki_2001}, and  even for arbitrarily small initial mass  in higher-dimensions~\cite{Nagai_2000, Herrero_Medina_Velazquez_1997}.
 A fundamentally different picture emerges when the chemotactic sensitivity depends on the signal concentration. When the sensitivity function takes a singular form such as $\chi(c)=\chi_{0}/c$,
  or more generally satisfies $\chi(s)\to 0$ as $s\to\infty$,  the decay of $\chi(c)$ at large $c$ can prevent singularity formation. Indeed, the authors in \cite{Fujie_2014} studied the parabolic--elliptic system$(\tau=0)$ with $\chi(c)=\chi_{0}/ c$
in a smoothly bounded domain $\Omega\subset\mathbb{R}^{2}$ and established global existence and uniform boundedness of solutions under a smallness condition on $\chi_{0}$. This result was generalized by Fujie and Senba~\cite{FujieSenba_DCDSB_2016}, who proved global existence and boundedness for the same parabolic--elliptic system with a general sensitivity $\chi>0$ satisfying  $\chi(s)\to 0$ as $s\to\infty$. 
In its fully parabolic version $(\tau>0)$, Fujie and Senba~\cite{fujie2016global} obtained global existence and boundedness results for radial solutions under the smallness $\tau\ll 1$, and subsequently proposed a sufficient condition on the sensitivity function that ensures boundedness of solutions to parabolic--parabolic Keller--Segel systems in higher-dimensionsal settings \cite{FujieSenba_Nonlinearity_2018}. 
Regarding long-time asymptotics under such a singular structure $\chi(c)=\chi_{0}/ c$, Ahn, Kang, and Lee~\cite{Ahn_Kang_Lee_2019} established eventual smoothness and exponential stabilization of global weak solutions for the parabolic--elliptic system in three and four dimensions.
When fluid coupling is further imposed, Winkler~\cite{Winkler2020} established global boundedness of small-mass solutions in the 2D Keller--Segel--Navier--Stokes system with a constant sensitivity, and
Black, Lankeit, and Mizukami~\cite{Black_Lankeit_Mizukami_2018} addressed the fluid-coupled Keller--Segel system with singular sensitivity $\chi(c)=\chi_{0}/c$ in bounded domains $\Omega\subset\mathbb{R}^{d}$, $d=2,3$, and proved global existence and boundedness under $\chi_{0}<\sqrt{2/d}$. 

In realistic biological environments, such as complex extracellular matrices or fluids, cell movement is often anisotropic, necessitating the modeling of chemotactic sensitivity as a tensor rather than a scalar~\cite{Bellomo_2015}. When the sensitivity $S$ is tensor-valued, a key mathematical difficulty arises, as one cannot exploit the cancellation structure that links the chemotactic flux $\nabla \cdot (nS\nabla c)$ with the energy-type estimates in the way that is possible for the scalar case~\cite{Winkler2012, Lankeit_Winkler_2023}. Under a small-data assumption, Cao and Lankeit~\cite{CaoLankeit_2016} obtained globally bounded solutions for a 3D chemotaxis--Navier--Stokes system with tensor-valued sensitivity. In the context of fluid-coupled Keller--Segel system with tensor-valued sensitivity satisfying a saturation condition $|S(x,n,c)|\le C_{S}(1+n)^{-\alpha}$, Wang and Xiang~\cite{WangXiang_2015,WangXiang_2016} established global existence and boundedness in two ($\alpha>0$) and three ($\alpha>1/2$) dimensions, and Winkler~\cite{Winkler_2018_fluid_saturated} relaxed this condition to  $\alpha>1/3$ in three dimensions.  
We note, however, that all the above results on fluid-coupled systems with tensor-valued sensitivity require either small-data assumptions or cell density-dependent saturation of the sensitivity. In particular, to the best of our knowledge, the question of global bounded solvability for Keller--Segel--Navier--Stokes systems with tensor-valued sensitivity and linear signal production  without any smallness conditions has remained open.

Motivated by these previous works, in this paper, we consider the following two-dimensional Keller--Segel--Navier--Stokes system involving a tensor-valued sensitivity $S$:
\begin{equation}\label{KS01}
\begin{cases}
n_t + u\cdot\nabla n = \Delta n - \nabla\cdot\bigl(n\,S(x,n,c)\nabla c\bigr), & \text{in } \Omega\times(0,\infty),\\
-\Delta c + u\cdot\nabla c + c = n, & \text{in } \Omega\times(0,\infty),\\
u_t + (u\cdot\nabla)u = \Delta u - \nabla\pi + n\nabla\Phi, & \text{in } \Omega\times(0,\infty),\\
\nabla\cdot u = 0, & \text{in } \Omega\times(0,\infty), \\
n(\cdot, 0)=n_0, \ u(\cdot, 0)=u_0, & \text{in } \Omega,
\end{cases}
\end{equation}
where $\Omega\subset\mathbb{R}^{2}$ is a bounded domain with smooth boundary $\partial\Omega$. 
Here, $n, c, u$, and $\pi$ denote the cell density, the chemical signal concentration, the fluid velocity, and the associated pressure, respectively. The function $\Phi$ represents a given external potential.

The boundary conditions are given by 
\begin{equation}\label{KS_boundary}
\bigl(\nabla n - n S(x,n,c)\nabla c\bigr)\cdot \nu =0, \qquad \nabla c \cdot \nu  = 0, \qquad u=0 \quad \text{on } \partial\Omega \times (0,\infty),
\end{equation}
where $\nu$ denotes the outward unit normal vector to $\partial \Omega$. Moreover, we assume that 
\begin{equation}\label{KS03}
\Phi\in C^{2}(\overline{\Omega}),\qquad n_0\in W^{1,\infty}(\Omega),\qquad 0 \not\equiv n_0\ge 0, \qquad u_0\in D(A^{\alpha})
\end{equation}
where $A^{\alpha}$, for $\alpha \in (\frac12,1)$, denotes the realization of the Stokes operator in the solenoidal subspace $L^{2}_{\sigma}(\Omega):= \bigl\{f\in L^{2}(\Omega)\,\big|\, \nabla \cdot f=0 \bigr\}$ of $L^{2}(\Omega)$. No initial datum is prescribed for $c$ because the equation $\eqref{KS01}_{2}$ is elliptic.

Let $\mathbb{R}_{+}:=(0,\infty)$. The sensitivity tensor $S=(S_{ij})_{i,j\in\{1,2\}}$ satisfies
\begin{equation}\label{KS_S}
S_{ij}\in C^{2}(\overline{\Omega}\times \overline{\mathbb{R}_{+}}\times \mathbb{R}_{+} )\quad \text{for all } i,j\in\{1,2\}, 
\end{equation}
and there exist constants $s_0>0$, $s_1\ge 0$, and $\gamma>0$ such that
\begin{equation}\label{KS02}
|S(x,r,s)| \le \frac{s_{0}}{(s_1+s)^{\gamma}} \quad \text{for } (x,r,s)\in \overline{\Omega}\times \overline{\mathbb{R}_{+}}\times \mathbb{R}_{+}.    
\end{equation}

The key novelty of this work lies in the treatment of the tensor-valued structure of $S$ under this singular or saturated decay condition~\eqref{KS02}. As discussed above, when the sensitivity is tensor-valued, the cancellation structure available for scalar sensitivities are no longer applicable.

To establish the global existence and boundedness of classical solutions to our system, we are motivated by the $\varepsilon$-regularity argument (e.g., Fujie and Senba \cite{FujieSenba_DCDSB_2016}). In fluid-free systems, this approach prevents blow-up by utilizing a uniform bound on $1/c$ when the local-in-space mass of $n$ is less than some $\varepsilon_{0}\ll1$, and otherwise by exploiting the weakened chemotactic drift locally in space due to the decay of the sensitivity. Unfortunately, due to the additional transport terms from the Navier--Stokes coupling, it is highly non-trivial to apply this exact $\varepsilon$-regularity idea. 
To overcome this difficulty, we modify the approach from Ahn, Kang, and Lee \cite{Ahn_Kang_Lee_2023}, which uses the  smallness of $\nabla c$ in $L^2_{\rm loc}(\Omega)$ to study a consumption-type parabolic--elliptic  system with tensor-valued sensitivities in a fluid-free setting. By combining this local smallness idea with an $\varepsilon$-regularity type argument, we successfully apply the method to our Keller-Segel--Navier--Stokes model under the power-decay condition. This method is also applicable to fluid-free systems with more general parameters $\gamma$ and $s_{1}$, which leads to the following result on the global existence and boundedness for both the fluid-coupled and fluid-free systems.
\begin{theorem}\label{THM1}
Let $\Omega\subset\R^2$ be a bounded smooth domain.

\smallskip
\noindent\textbf{(i) Fluid-coupled case.}
Assume that $\gamma>\frac12$ and $s_1>0$.
Then \eqref{KS01}--\eqref{KS02} admits a unique global classical solution $(n,c,u,\pi)$ such that
\[
n\in C([0,\infty);C(\overline{\Omega}))\cap C^{2,1}(\overline{\Omega}\times(0,\infty)),
\qquad
c\in C^{2,0}(\overline{\Omega}\times(0,\infty)),
\]
\[
u\in C([0,\infty);D(A^\alpha))\cap C^{2,1}(\overline{\Omega}\times(0,\infty)),
\qquad
\pi\in C^{1,0}(\overline{\Omega}\times(0,\infty)),
\]
and moreover
\[
n\in L^\infty(0,\infty;L^\infty(\Omega)),\quad
c\in L^\infty(0,\infty;W^{1,\infty}(\Omega)),\quad
u\in L^\infty(0,\infty;L^\infty(\Omega)).
\]
In particular, there exists $C=C (\Omega,\alpha,s_0,s_1,\gamma,\int_\Omega n_0, \|n_0\|_{L^{\infty}(\Omega)},\|u_0\|_{L^2(\Omega)},\|A^\alpha u_0\|_{L^2(\Omega)},\|\nabla\Phi\|_{L^\infty(\Omega)} )>0$
such that
\[
\sup_{t>0}\Bigl(\|n(\cdot,t)\|_{L^\infty(\Omega)}+\|c(\cdot,t)\|_{W^{1,\infty}(\Omega)}+\|u(\cdot,t)\|_{L^\infty(\Omega)}\Bigr)\le C.
\]

\smallskip
\noindent\textbf{(ii) Fluid-free case.}
Assume that $\gamma>0$, $s_1\ge0$, and the system is fluid-free (i.e. $u\equiv 0 \equiv \nabla \Phi$).
Then the reduced system admits a unique global classical solution $(n,c)$ such that
\[
n\in C([0,\infty);C(\overline{\Omega}))\cap C^{2,1}(\overline{\Omega}\times(0,\infty)), \quad c\in C^{2,0}(\overline{\Omega}\times(0,\infty))
%\cap L^\infty(0,\infty;L^\infty(\Omega)),
\]
%\[
%c\in C^{2,0}(\overline{\Omega}\times(0,\infty))
%\cap L^\infty(0,\infty;W^{1,\infty}(\Omega)), 
%\]
and moreover
\[
n\in L^\infty(0,\infty;L^\infty(\Omega)), \quad 
c\in L^\infty(0,\infty;W^{1,\infty}(\Omega)).
\]
In particular, there exists $C=C (\Omega,s_0,s_1,\gamma,\int_\Omega n_0, \|n_0\|_{L^{\infty}(\Omega)} )>0$
such that
\[
\sup_{t>0}\Bigl(\|n(\cdot,t)\|_{L^\infty(\Omega)}+\|c(\cdot,t)\|_{W^{1,\infty}(\Omega)}\Bigr)\le C.
\]
\end{theorem}

The proof of Theorem~\ref{THM1} mainly relies on a sequence of localized energy estimates. Because our model is of production type, obtaining the smallness of $\nabla c$ in $L^2_{\rm loc}(\Omega)$ as in \cite{Ahn_Kang_Lee_2023} is no longer feasible. Instead, we use the power-decay condition \eqref{KS02} of the tensor-valued sensitivity to establish the $L^{2}_{\rm loc}(\Omega)$-smallness of the weighted gradient $\frac{\nabla c}{(s_1 + c)^{(\beta+1)/2}}$ for $\beta>1$. 
This local smallness is used to obtain uniform bounds for the localized $L \ln L$-norm of $n$. By using covering arguments, these local bounds provide global-in-space estimates of $n$. These global estimates improve the regularity of $c$ and $u$, which subsequently increases the regularity of $n$. This makes it possible to apply a Moser-type iteration, eventually yielding the uniform global boundedness.

In Theorem \ref{THM3}, we establish the exponential-in-time stabilization of solutions for the fluid-free system. By establishing an interpolation inequality involving the H\"older norm (Lemma~\ref{EHRLEM}), we prove exponential convergence when either the symmetric part of $S$ is negative semi-definite or $S$ is an isotropic tensor with a small coefficient.
\begin{theorem}\label{THM3}
Let $\Omega\subset \R^{2}$ be a bounded smooth domain. Assume that $\gamma>0$, $s_1\ge0$, and the system is fluid-free (i.e. $u\equiv 0 \equiv \nabla \Phi$).
 Then, the solution $(n,c)$ obtained in Theorem~\ref{THM1} satisfies 
\[
 n(\cdot,t)\rightarrow \frac{1}{|\Omega|}\int_{\Omega}n_{0} \ \mbox{ in }L^{\infty}(\Omega),\qquad   c(\cdot,t)\rightarrow \frac{1}{|\Omega|}\int_{\Omega}n_{0}  \ \mbox{ in } W^{1,\infty}(\Omega)
\]
exponentially in time provided that $S$ satisfies one of the following:
\begin{enumerate}[(i)]
\item $\hat{S}:=\frac{S+S^{\intercal}}{2}$ is negative semi-definite (i.e. ${\rm tr} (\hat{S})\le 0$ and ${\rm det} (\hat{S})\ge 0$).
\smallskip
\item $S=\frac{s_{0}}{c^{\gamma}} \mathbb{I}$, $\gamma\ge1$, $s_{0}< s_{\star}$, and $\Omega$ is convex, where $s_{\star}=s_{\star}(\gamma,\Omega, \int_{\Omega}n_{0})>0$ is specified in \eqref{SSTAR}.
\end{enumerate}
\end{theorem}

Our paper is organized as follows. Section~\ref{S:Prelim} is devoted to preliminaries, including the local existence of solutions. In Section~\ref{S:Proof_Thm1}, we establish key estimates and provide the proof of Theorem~\ref{THM1}. Finally, in Section~\ref{S:Asympt}, we prove Theorem~\ref{THM3} on exponential-in-time stabilization.

%%%
\section{Preliminary}\label{S:Prelim}

We first recall that \eqref{KS01}--\eqref{KS02} admits a unique local-in-time classical solution.
%The problem  \eqref{KS01}--\eqref{KS02} admits local-in-time classical solutions.

\begin{lemma}\label{LWLEM} 
Let $\Omega \subset \R^{2}$ be a bounded smooth domain. Then, there exists a maximal time of existence $T_{\rm max}\in(0,\infty]$ and functions
\[
n\in C([0,T_{\rm max});C(\overline{\Omega}))\cap C^{2,1}(\overline{\Omega}\times(0,T_{\rm max})),
\qquad
c\in  C^{2,0}(\overline{\Omega}\times(0,T_{\rm max})),
\]
\[
u\in  C([0,T_{\rm max});D(A^{\alpha})) \cap C^{2,1}(\overline{\Omega}\times(0,T_{\rm max})),\qquad \pi  \in C^{1,0}(\overline{\Omega}\times(0,T_{\rm max})),
\]
 such that $n>0$ and $c>0$ in $\overline{\Omega}\times(0,T_{\rm max})$, that $(n,c,u,\pi)$  solves  \eqref{KS01}--\eqref{KS02} in the classical sense in $\Omega\times(0,T_{\rm max})$, and that
 \[
 \mbox{ if } \quad T_{\rm max}<\infty, \quad \mbox{ then }\quad  \limsup_{t\nearrow T_{\rm max}}\|n(\cdot,t)\|_{L^{\infty}}=\infty. 
 \]
Moreover, 
\[
\int_{\Omega}n(\cdot,t)=\int_{\Omega}n_{0}=\int_{\Omega}c(\cdot,t),\qquad t<T_{\rm max},
\]
and for any $\sigma_{1}\ge0$ and $r\in\R$,
\begin{equation}\label{NABLACL2}
r\int_{\Omega}\frac{|\nabla c|^{2}}{(\sigma_{1}+c)^{r+1}}(\cdot,t)+\int_{\Omega}\frac{n}{(\sigma_{1}+c)^{r}}(\cdot,t)=\int_{\Omega}\frac{c}{(\sigma_{1}+c)^{r}}(\cdot,t),\qquad t<T_{\rm max}.
\end{equation}
\end{lemma}

\begin{proof}
The local existence and basic properties can be obtained by directly adapting standard theories for Keller--Segel systems. We only detail the proof of the identity \eqref{NABLACL2}.

Multiplying \eqref{KS01}$_{2}$ by $(\sigma_{1}+c)^{-r}$ for  $\sigma_{1}\ge 0$ and $r\in\mathbb{R}$, and integrating over $\Omega$, we obtain 
\begin{equation}\label{eq:NABLA-start}
\int_{\Omega} \frac{u\cdot\nabla c}{(\sigma_{1}+c)^{r}}
-\int_{\Omega} \frac{\Delta c}{(\sigma_{1}+c)^{r}}
= \int_{\Omega}\frac{n}{(\sigma_{1}+c)^{r}}-\int_{\Omega}\frac{c}{(\sigma_{1}+c)^{r}}.
\end{equation}
Then the following term vanishes, since $\nabla \cdot u =0$ in $\Omega$ and $u=0$ on $\partial \Omega$, that
\[
\int_{\Omega} \frac{u\cdot\nabla c}{(\sigma_{1}+c)^{r}} = \int_{\Omega} u \cdot \nabla G(c) =0, \quad \mbox{where } G(c):= \int_{0}^{c} (\sigma_1 + s)^{-r} \,ds.
\]
Moreover, by integration by parts with $\partial_{\nu} c =0$ on $\partial \Omega$, we have 
\[
-\int_{\Omega} \frac{\Delta c}{(\sigma_{1}+c)^{r}} = -r\int_{\Omega} \frac{|\nabla c|^2}{(\sigma_1 + c)^{r+1}}. 
\] 
Inserting them into \eqref{eq:NABLA-start} yields the desired equality \eqref{NABLACL2}. 
\end{proof}

The following uniform-in-time lower estimate can be found in \cite[Lem.~2.1]{Fujie_2014}. We will use Lemma \ref{ULE} only in the fluid-free case where the $c$-equation reduces to $-\Delta c+c=n$.
\begin{lemma}\label{ULE}
Let $\Omega\subset \R^{d}$, $d\ge1$, be a smoothly bounded domain, and let $w\in  C(\overline{\Omega})$ be a non-negative function such that $\int_{\Omega} w \,dx > 0$. If $z$ is a $ C^{2}(\overline{\Omega})$ solution to
\begin{equation}
\begin{cases}
-\Delta z + z = w & x\in\Omega, \\
\frac{\partial z}{\partial\nu} = 0 & x\in\partial\Omega,
\end{cases}
\end{equation}
then there is a positive constant $C=C(d,\Omega)>0$ such that $z$ satisfies point-wise estimate
\begin{equation}
z \ge C \int_{\Omega} w \,dx > 0 \quad \text{in } \Omega.
\end{equation}
\end{lemma}

We recall from \cite[Lem.~2.2]{Winkler2020} the following Trudinger-Moser type inequality.
\begin{lemma}\label{TMLEM}  
Let $\Omega\subset \mathbb{R}^{2}$ be a bounded smooth domain. For any choice of $\lambda>0$, there exists $C=C(\Omega,\lambda)>0$ such that if $0\not\equiv f\in C(\overline{\Omega})$ is nonnegative and $g\in W^{1,2}(\Omega)$, then for each $a>0$,
\[
\int_{\Omega}f|g|\le \frac{1}{a}\int_{\Omega} f\ln\left(\frac{f}{\overline f}\right)+\frac{(1+\lambda)a}{8\pi}\int_{\Omega}f\int_{\Omega}|\nabla g|^{2}+Ca\int_{\Omega}f\left(\int_{\Omega}|g|\right)^{2}\!\!+\frac{C}{a}\int_{\Omega}f,
\]
where $\overline{f}=\frac{1}{|\Omega|}\int_{\Omega} f$.
\end{lemma}
 Recall  the following lemma from \cite[Lem.2.4]{fujie2016global}.
\begin{lemma}\label{LNIMPR}
Let $\Omega\subset \mathbb{R}^{2}$ be a bounded smooth domain.
 There exists $C=C(\Omega)>0$ such that
for any $p\ge1$, $s>1$, $\varepsilon>0$, and non-negative $f\in C^{1}(\overline{\Omega})$,
\[
\begin{aligned}
\int_{\Omega}f^{p+1}&\le C\frac{(p+1)^{2}}{\ln s}\int_{\Omega} (f\ln f+e^{-1})\int_{\Omega}f^{p-2}|\nabla f|^{2}
+(4C)^{1+\frac{\varepsilon}{2}}\bke{\int_{\Omega}f^{\frac{\varepsilon}{2}\frac{p+1}{1+\varepsilon}}  }^{\frac{2(1+\varepsilon)}{\varepsilon}}+6s^{p+1}|\Omega|.
\end{aligned}
\]
\end{lemma}
%\begin{proof}
%Note that if $0\not\equiv f\in C(\overline{\Omega})$ is nonnegative, there is the result
% \cite[Lem.~2.2]{Winkler2020}. To extend it to the case $0\not\equiv f\in L \ln L(\Omega)$ is nonnegative, we use approximation $f_{\varepsilon}=e^{\varepsilon\Delta }f$ with Neumann heat kernel $e^{\varepsilon\Delta }$?  
%\end{proof}
%Recall the standard extension lemma for $f\in W^{1,p}(\Omega)$, $1\le p \le \infty$ (see, e.g., Evan's PDE)
%\begin{lemma}\label{EXLEM} (Extension lemma)
%Let $\Omega\subset \R^{2}$ be a bounded smooth domain and $1\le p \le \infty$. Let $D$ be a bounded open set such that $\Omega \subset \subset D$. Then, there exists an operator
%\[
%E\,:\,W^{1,p}(\Omega )\rightarrow W^{1,p}(\R^{d})
%\]
%such that for each $f\in W^{1,p}(\Omega )$:
%\begin{itemize}
%\item[(i)] $Ef=f$ a.e. in $\Omega$,
%\item[(ii)] $Ef$ has support within $D$,
%\item[(iii)] $\|Ef\|_{W^{1,p}(\R^{d})}\le C \|Ef\|_{W^{1,p}(\Omega)}$,
%\end{itemize}
%where $C$ depending only on $p$, $\Omega$, and $D$.
%\end{lemma}
Note also  from \cite[Lem.~3.4.]{Winkler2019} the following lemma. 
\begin{lemma}\label{EXPTINT}
If the nonnegative $h\in L^{1}_{\rm loc}(\mathbb{R})$ has the property that there exist $\tau>0$ and $b>0$ such that
\[
\frac{1}{\tau}\int_{t}^{t+\tau}h(s)ds\le b,\qquad t\in (t_0,T-\tau),
\]
and if
$y\in C([t_0,T))\cap C^{1}(t_0,T)$ satisfies 
\[
y'(t)+ay(t)\le h(t)
\]
 with some $a>0$, then
\[
y(t)\le y(t_{0})+\frac{b\tau}{1-e^{-a\tau}}\quad\mbox{ for all }t\in(t_{0},T).
\]
\end{lemma}

For a spatially localized estimate, we prepare the following cut-off function.
\begin{lemma}\label{PSILEM}
Let $d\ge2$,  $\eta\in(0,e^{-1})$ and  $B_{\eta}(0)=\{ x\in\mathbb{R}^{d}\,|\, |x|<\eta\}$.
The non-negative radial   function 
\[
\psi_{\eta}(x):=
\left\{  
\begin{array}{ll}
\ln(-\ln |x| )-\ln(-\ln\eta),\qquad &x\in B_{\eta}(0)\setminus\{0\},\vspace{1mm}
\\
0,  &{\rm otherwise},
\end{array}
\right.
\]
satisfies $\psi_{\eta}\nearrow \infty$  as $r:=|x|\rightarrow 0$, and 
\[
\|\psi_{\eta}\|_{H^{1}(\mathbb{R}^{d})}+\|\nabla  \psi_{\eta}^{2} \|_{L^{2}(\mathbb{R}^{d})}\rightarrow 0\quad\mbox{ as }\quad\eta\rightarrow 0. 
\]
 \end{lemma}
\begin{proof}
In \cite[Prop.~3.1.]{Ahn_Kang_Lee_2023}, $
\|\psi_{\eta}\|_{H^{1}(\mathbb{R}^{d})}\rightarrow 0$  as $\eta\rightarrow 0$ has been shown. It is easy to see that $\psi_{\eta}\nearrow \infty$  as $|x|\rightarrow 0$. 
Now, it remains to show
\[
\|\nabla \psi_{\eta}^{2}\|_{L^{2}(\mathbb{R}^{d})}\rightarrow 0\quad\mbox{ as }\quad\eta\rightarrow 0. 
\]
We denote the surface area of $B_{1}(0)$  as  $\sigma_{d}$.
A direct computation yields that   
\begin{align*}
\begin{aligned}
\|\nabla \psi_{\eta}^{2}\|_{L^{2}(\mathbb{R}^{d})}^{2}=\|\partial_{r} (\psi_{\eta}^{2})\|_{L^{2}(\mathbb{R}^{d})}^{2}
&=4\sigma_{d}\int_{0}^{\eta} \psi_{\eta}^{2} \psi_{\eta r}^{2}r^{d-1}dr
\\&\le 4\sigma_{d}\int_{0}^{\eta}\frac{|\ln(-\ln r)|^{2}}{|r\ln r|^{2}}r^{d-1}dr 
\\&=4\sigma_{d}\int_{-\infty}^{\ln\eta}\frac{|\ln(-z)|^{2}}{z^{2}}e^{(d-2)z}dz
\\&=4\sigma_{d}\int_{\ln(-\ln\eta)}^{\infty} \frac{\tau^{2}}{e^{\tau}}e^{-(d-2)e^{\tau}}d\tau
\\&\le 4\sigma_{d}\int_{\ln(-\ln\eta)}^{\infty} \frac{\tau^{2}}{e^{\tau}}d\tau
\end{aligned}
\end{align*}
and the rightmost term tends to $0$ as $\eta\rightarrow 0$. This completes the proof.
\end{proof}

%%%
\section{Global Existence and Boundedness: Proof of Theorem~\ref{THM1}}\label{S:Proof_Thm1}

In this section, we prove Theorem~\ref{THM1}. We begin with key integral estimates derived from \eqref{KS01}.

In the subsequent lemmas and proofs, a constant $C>0$ denotes a generic constant which may vary from line to line and depends only on the prescribed data.

\begin{lemma}\label{LEMNABLNN}
Let $\gamma>\frac{1}{2}$ and $s_{1}>0$.
There exists $C=C(\Omega, \gamma, s_{0}, s_{1}, \int_{\Omega}n_{0})>0$ satisfying
\[
\int_{t}^{t+\tau}\int_{\Omega} |\nabla \ln(n+1)|^{2}\le C\quad\mbox{ for all }t\in(0,T_{\rm max}-\tau),
\]
where $\tau:=\min\{1,\frac{1}{2}T_{\rm max}\}$.
\end{lemma}
\begin{proof}
We divide $\eqref{KS01}_{1}$ by $n+1$ and integrate over $\Omega$, then it follows by \eqref{KS02} and Young's inequality,  
%Using $\eqref{KS01}_{1}$ and Young's inequality, we compute $-\frac{d}{dt}\int_{\Omega}\ln(n+1)$  as
\[
\begin{aligned}
-\frac{d}{dt}\int_{\Omega}\ln(n+1)+\int_{\Omega} |\nabla \ln(n+1)|^{2}
&\le s_{0}\int_{\Omega} \frac{n |\nabla n| |\nabla c|}{(n+1)^{2}(s_{1}+c)^{\gamma}}
\\&\le \frac{1}{2}\int_{\Omega}|\nabla \ln(n+1)|^{2}+\frac{s_{0}^{2}}{2}\int_{\Omega}\frac{|\nabla c|^{2}}{(s_{1}+c)^{2\gamma}}.
\end{aligned}
\]
The first term on RHS is absorbed by the second term on LHS. 
Integrating this inequality over $[t,t+\tau]$ yields the following, 
\[
\begin{aligned}
\int_{t}^{t+\tau}\int_{\Omega} |\nabla \ln(n+1)|^{2}
&\le 2\int_{\Omega}\ln(n+1)(\cdot,t+\tau)+  \frac{s_{0}^{2}}{2}\int_{t}^{t+\tau}\int_{\Omega}\frac{|\nabla c|^{2}}{(s_{1}+c)^{2\gamma}}
\\&\le 2\int_{\Omega}n_{0}+ \frac{s_{0}^{2}}{2}\int_{t}^{t+\tau}\int_{\Omega}\frac{|\nabla c|^{2}}{(s_{1}+c)^{2\gamma}},
\end{aligned}
\]
since $\ln (a+1) \le a$ for all $a\ge0$ and $\int_{\Omega}n(\cdot, t)=\int_{\Omega}n_{0}$. 
Finally, because $\gamma>\frac{1}{2}$ and $s_{1}>0$, choosing $\sigma_1 = s_1$ and $r=2\gamma -1 >0$ in  \eqref{NABLACL2}, and $\int_{\Omega}c=\int_{\Omega}n_{0}$ gives 
\[
(2\gamma -1)\int_{\Omega}\frac{|\nabla c|^{2}}{(s_{1}+c)^{2\gamma}}
\le \int_{\Omega}\frac{c}{(s_{1}+c)^{2\gamma-1}}
\le \frac{1}{s_{1}^{2\gamma-1}} \int_{\Omega} c 
\le \frac{1}{s_{1}^{2\gamma-1}} \int_{\Omega}n_{0}.
\]
This entails the desired result. 
\end{proof}

The combination of a uniform estimate Lemma~\ref{LEMNABLNN} and Trudinger--Moser type inequality in Lemma~\ref{TMLEM}, we obtain a uniform bound for the space-time integrability of $n\ln(n+1)$.

\begin{lemma}\label{LEMLLNL}
Let $\gamma>\frac{1}{2}$ and $s_{1}>0$.
There exists $C=C(\Omega, \gamma, s_{0}, s_{1}, \int_{\Omega}n_{0})>0$ satisfying
\[
\int_{t}^{t+\tau}\int_{\Omega} n \ln(n+1) \le C\quad\mbox{ for all }t\in(0,T_{\rm max}-\tau),
\]
where $\tau:=\min\{1,\frac{1}{2}T_{\rm max}\}$.
\end{lemma}

\begin{proof}
Note that $\int_{\Omega} n (\cdot, t) = \int_{\Omega} n_0$. Apply Lemma~\ref{TMLEM} with $f=n$, $g=\ln(n+1)$, and $a=\lambda=2$. Then we have the following, for $\overline{n}=\frac{1}{|\Omega|}\int_\Omega n$, that 
\[
\begin{aligned}
\int_{\Omega} n \ln(n+1)
&\le \frac{1}{2}\int_{\Omega} n\ln n-\frac{1}{2}(\ln\overline{n})\int_{\Omega} n + \frac{3}{4\pi} \int_{\Omega}n \int_{\Omega}|\nabla \ln(n+1)|^{2}
\\&
\quad+2C\int_{\Omega}n\left(\int_{\Omega}\ln(n+1)\right)^{2}+\frac{C}{2}\int_{\Omega}n.
\end{aligned}
\]
Using $a\ln a\le a\ln (a+1)$, for $a\geq 0$, the first term on RHS is absorbed to the LHS. Also, using $\ln (a+1)\le a$, for $a\ge0$ and the mass conservation property of $n$ yields that there is a constant $C>0$ such that 
\[
\int_{\Omega} n \ln(n+1)
\le C\int_{\Omega}|\nabla \ln(n+1)|^{2} + C.
\]
We complete the proof by integrating this inequality over $[t,t+\tau]$ and using Lemma~\ref{LEMNABLNN}. 
\end{proof}

With the $L^1$-bound for $n\ln(n+1)$ Lemma~\ref{LEMLLNL}, we show that $u$ belongs to the energy class $L^\infty_t L^2_x \cap L^2_t W^{1,2}_x$.

\begin{lemma}\label{UBDD}
Let $\gamma>\frac{1}{2}$ and $s_{1}>0$. 
There exists $C=C(\Omega, \gamma, s_0, s_1, \int_{\Omega}n_0, \|u_0\|_{L^2(\Omega)}, \|\nabla \Phi\|_{L^{\infty}(\Omega)})>0$ such that
\begin{equation}\label{UBDD01}
\sup_{t<T_{\rm max}}\int_{\Omega}|u(\cdot,t)|^{2}\le C,
\end{equation}
and 
\begin{equation}\label{UBDD02}
\int_{t}^{t+\tau}\int_{\Omega}|\nabla u|^{2}\le  C\quad\mbox{ for all }\quad t\in(0,T_{\rm max}-\tau),
\end{equation}
where $\tau:=\min\{1,\frac{1}{2}T_{\rm max}\}$.
\end{lemma}

\begin{proof}
Testing $\eqref{KS01}_{3}$ by $u$ and integrating over $\Omega$, it follows that, using $\nabla\cdot u=0$ in $\Omega$ and $u=0$ on $\partial\Omega$, 
\[
\frac12\frac{d}{dt}\int_\Omega |u|^2 + \int_\Omega |\nabla u|^2
= \int_\Omega n\,\nabla\Phi\cdot u
\le \|\nabla\Phi\|_{L^\infty(\Omega)}\int_\Omega n|u|.
\] 
Applying Lemma~\ref{TMLEM} with $f=n$ and $g=|u|$, $\lambda =1$ and $a>0$ to be chosen later, we have
\[
\int_{\Omega}n |u|
\le   \frac{1}{a}\int_{\Omega} n\ln\left(\frac{n}{\overline{n_0}}\right)+\frac{a}{4\pi}\int_{\Omega}n_0\int_{\Omega}|\nabla u|^{2}+Ca\int_{\Omega}n_0\left(\int_{\Omega}|u|\right)^{2}+\frac{C}{a}\int_{\Omega}n_0.
\]
for a positive constant $C$. Moreover, by H\"{o}lder inequality and Poincar\'e's inequality, we have 
\[
\left(\int_\Omega |u|\right)^2 \le |\Omega|\,\|u\|_{L^2(\Omega)}^2 \le C(\Omega)\int_\Omega |\nabla u|^2.
\]
Moreover, as in the proof of Lemma~\ref{LEMLLNL}, there is $C=C(\Omega, \int_{\Omega}n_0)>0$ such that  
\[
\int_\Omega n\ln\left(\frac{n}{\overline{n_0}}\right)
\le \int_\Omega n\ln(n+1) + C.
\]
Combination of all estimates above using the mass conservation property of $n$ yields the following 
\[
\frac12\frac{d}{dt}\int_\Omega |u|^2 + \int_\Omega |\nabla u|^2
\le \|\nabla\Phi\|_{L^{\infty}(\Omega)} 
\left[
\frac{1}{a}\int_\Omega n\ln(n+1) + C a\,\int_{\Omega}n_0 \int_\Omega |\nabla u|^2 + \frac{C}{a}
\right].
\]
Choose $a>0$ sufficiently small that $\|\nabla\Phi\|_{L^{\infty}(\Omega)} C a\,\int_{\Omega}n_0 <\frac{1}{2}$, then we can absorb the term involving $\int_{\Omega}|\nabla u|^2$ into the LHS and have 
\begin{equation}\label{MAR63}
\frac{d}{dt}\int_\Omega |u|^2 + \frac{1}{C}\int_\Omega |\nabla u|^2
\le C\left(1+\int_\Omega n\ln(n+1)\right),
\end{equation}
for a constant $C=C(\Omega, \int_{\Omega} n_0, \|\nabla \Phi\|_{L^{\infty}(\Omega)})>0$. 
By Poincar\'e's inequality, $\int_\Omega |\nabla u|^2 \ge C(\Omega)\int_\Omega |u|^2$, and Lemma~\ref{LEMLLNL}, a uniform bound \eqref{UBDD01} follows by applying Lemma~\ref{EXPTINT} with $C$ additionally depending on $\|u_0\|^{2}_{L^2(\Omega)}$. 
Integrating \eqref{MAR63} over $[t,t+\tau]$ and using \eqref{UBDD01} together with Lemma~\ref{LEMLLNL}, we obtain \eqref{UBDD02}.
\end{proof}

 Next, we establish a uniform integrability property for the weighted gradient density
\[
\frac{|\nabla c(\cdot,t)|^{2}}{(s_{1}+c(\cdot,t))^{\beta+1}}
\quad \mbox{for } t\in(0,T_{\max})  \mbox{ and }  \beta>1.
\]
More precisely, we show that this family is equi-integrable in space, uniformly in time,
which will later rule out concentration of the weighted energy near points.

%Next, we prove the equi-integrability of 
%\[
%\bket{\int_{\Omega} \frac{|\nabla  c|^{2}}{(s_{1}+c)^{\beta+1}}(\cdot,t) \,\Bigr{|}\, t\in(0,T_{\rm max})}
%\]
%when $\beta>1$.

\begin{proposition}\label{KEYPRO}   Let $\beta>1$. Assume that either $s_{1}>0$, or $s_{1}\ge0$ and the system is fluid-free. Then, for any $\varepsilon>0$, there exists $\delta_{\varepsilon}\in(0, e^{-1})$ independent of $q\in\overline{\Omega}$ such that
\[
\sup_{t<T_{\rm max}}\int_{\Omega\cap B_{\delta}(q)}\frac{|\nabla c|^{2}}{(s_{1}+c)^{\beta+1}} (\cdot,t)\le \varepsilon\quad\mbox{ for }\quad \delta\in(0,\delta_{\varepsilon}).
\]
\end{proposition}

\begin{proof}
Fix $q\in\overline{\Omega}$ and let $\psi(x):=\psi_\eta(x-q)$, where $\psi_\eta$ is given by Lemma~\ref{PSILEM} (with $d=2$) for some $\eta\in(0,e^{-1})$.
Note that $\supp\psi\subset B_\eta(q)$ and let $B_{\eta}:= B_{\eta}(q)$. Hence all integrals below are taken over $\Omega\cap B_\eta$.
Multiply \eqref{KS01}$_2$ by $\frac{-1}{\beta(s_{1}+c)^{\beta}} \psi^2$ and integrating over $\Omega$, we obtain 
	\begin{equation}\label{KEYPRO_E}
	\begin{aligned}
	   & \int_{\Omega \cap B_{\eta}} \frac{|\nabla c|^2}{(s_{1}+c)^{\beta+1}} \psi^2 
		       + \frac{1}{\beta}\int_{\Omega \cap B_{\eta}}\frac{n}{(s_{1}+c)^{\beta}} \psi^2\\
		       &\quad = \frac{1}{\beta}\int_{\Omega \cap B_{\eta}}\frac{c}{(s_{1}+c)^{\beta}} \psi^2 
		       + \frac{1}{\beta}\int_{\Omega \cap B_{\eta}}\frac{u \cdot \nabla c}{(s_{1}+c)^{\beta  }} \psi^2
		       - \frac{2}{\beta}\int_{\Omega \cap B_{\eta}} \frac{\psi \nabla c \cdot \nabla \psi}{(s_{1}+c)^{\beta}}. 
	\end{aligned}
	\end{equation}
Then we observe that 
	\[
	\frac{1}{\beta}\int_{\Omega \cap B_{\eta}}\frac{c}{(s_{1}+c)^{\beta}} \psi^2 \leq \frac{1}{\beta}\int_{\Omega \cap B_{\eta}}  \frac{1}{(s_{1}+c)^{\beta-1}}\psi^2.
	\]
Moreover, by the Young's inequality,
	\[
	- \frac{2}{\beta}\int_{\Omega \cap B_{\eta}} \frac{\psi \nabla c \cdot \nabla \psi}{(s_{1}+c)^{\beta}} \leq \frac{1}{2}\int_{\Omega \cap B_{\eta}}\frac{|\nabla c|^2}{(s_{1}+c)^{\beta+1}} \psi^2 + \frac{2}{\beta^2}\int_{\Omega \cap B_{\eta}}\frac{1}{(s_{1}+c)^{\beta-1}} |\nabla \psi|^2.
	\]
Inserting these bounds into \eqref{KEYPRO_E} and discarding the nonnegative $n$-term yields
\begin{equation}\label{KEYPRO_E1}
\frac12\int_{\Omega}\frac{|\nabla c|^2}{(s_1+c)^{\beta+1}}\psi^2
\le \frac{1}{\beta}\int_{\Omega}\frac{1}{(s_1+c)^{\beta-1}}\psi^2
+\frac{1}{\beta}\int_{\Omega}\frac{u\cdot\nabla c}{(s_1+c)^{\beta}}\psi^2
+\frac{2}{\beta^2}\int_{\Omega}\frac{|\nabla\psi|^2}{(s_1+c)^{\beta-1}}.
\end{equation}

$\bullet$ \emph{Case 1: $s_1 >0$.}  Note that 
	\[
	\frac{1}{\beta}\int_{\Omega \cap B_{\eta}}\frac{u \cdot \nabla c}{(s_{1}+c)^{\beta  }} \psi^2 = \frac{1}{\beta}\int_{\Omega \cap B_{\eta}} u\cdot \nabla h(c) \psi^2, 
	\ \text{ where } \ h(c)=\int_{0}^{c} (s_{1}+\sigma)^{-\beta} \,d\sigma.
	\]
Because 
 	\[
	0 \le h(c) = \frac{1}{\beta -1} \left[s_{1}^{1-\beta}-(s_{1}+c)^{1-\beta}\right] \leq \frac{s_{1}^{1-\beta}}{\beta - 1},
	\]
it follows that, by the integration by parts using $\nabla\cdot u=0$ in $\Omega$ and $u=0$ on $\partial\Omega$ and H\"{o}lder inequality,
	\[\begin{aligned}
	\frac{1}{\beta}\int_{\Omega \cap B_{\eta}}\frac{u \cdot \nabla c}{(s_{1}+c)^{\beta  }} \psi^2 
	& = -\frac{1}{\beta}\int_{\Omega \cap B_{\eta}} h(c)   u \cdot \nabla \psi^{2} \\
	&\leq \frac{s_{1}^{1-\beta}}{\beta (\beta - 1)} \|u\|_{L^{2}(\Omega)}\bke{\int_{\Omega \cap B_{\eta}}    |\nabla \psi^{2}|^{2}}^{\frac{1}{2}}.
	\end{aligned}
	\]
Since $\beta >1$, combination of all estimates above yields that 
\[\begin{aligned}
	\int_{\Omega \cap B_{\eta}} \frac{|\nabla c|^2}{(s_{1}+c)^{\beta+1}} \psi^2 
		       %+ \frac{2}{\beta}\int_{\Omega \cap B_{\eta}}\frac{n}{\beta(s_{1}+c)^{\beta}} \psi^2 \\
     \leq \frac{2}{\beta s_{1}^{\beta-1}}\int_{\Omega \cap B_{\eta}}   \psi^2 + \frac{4}{\beta^2s_{1}^{\beta-1}}\int_{\Omega \cap B_{\eta}}  |\nabla \psi|^2 + \frac{2 s_{1}^{1-\beta}}{\beta (\beta - 1)} \|u\|_{L^{2}(\Omega)}\left(\int_{\Omega \cap B_{\eta}}|\nabla \psi^2|^2\right)^{\frac{1}{2}},
\end{aligned}
\]
and thus, we can conclude the desired result from Lemma~\ref{PSILEM} and Lemma~\ref{UBDD}.
\smallskip 

$\bullet$ \emph{Case 2: $s_{1}\ge0$ and the system is fluid-free.}  In this case, the second term on RHS of \eqref{KEYPRO_E} is absent. Moreover, by Lemma~\ref{ULE} and $\int_{\Omega}n=\int_{\Omega}n_{0}$, there exists $C>0$ satisfying  $c(x) \ge C\int_{\Omega}n_{0}$ for $x\in \Omega$. Thus, we have
	\[
	\frac{1}{\beta}\int_{\Omega \cap B_{\eta}}\frac{c}{(s_{1}+c)^{\beta}} \psi^2 \le \frac{1}{\beta (C\int_{\Omega}n_{0})^{\beta-1}}\int_{\Omega \cap B_{\eta}}  \psi^2,
	\]
	and 
	\[
	 - \frac{2}{\beta}\int_{\Omega \cap B_{\eta}} \frac{\psi \nabla c \cdot \nabla \psi}{(s_{1}+c)^{\beta}} \leq \frac{1}{2}\int_{\Omega \cap B_{\eta}}\frac{|\nabla c|^2}{(s_{1}+c)^{\beta+1}} \psi^2 + \frac{2}{\beta^2(C\int_{\Omega}n_{0})^{\beta-1}}\int_{\Omega \cap B_{\eta}}|\nabla \psi|^2.
	\]
This yields 
\[\begin{aligned}
	\int_{\Omega \cap B_{\eta}} \frac{|\nabla c|^2}{(s_{1}+c)^{\beta+1}} \psi^2 
    \leq \frac{2}{\beta (C\int_{\Omega}n_{0})^{\beta-1}}\int_{\Omega \cap B_{\eta}}   \psi^2 + \frac{4}{\beta^2(C\int_{\Omega}n_{0})^{\beta-1}}\int_{\Omega \cap B_{\eta}}  |\nabla \psi|^2.
     \end{aligned}
\]
Therefore, as in the previous case, we have the desired result using Lemma~\ref{PSILEM}. 
\end{proof}

We now turn to localized energy estimates by introducing a standard smooth cut-off function.

\begin{lemma}\label{LEMSTEST}
Let $\delta >0$. There exists a radially decreasing function 
$\varphi_{\delta} \in C_{0}^{\infty}(\mathbb{R}^d)$, $d\ge1$,  such that 
	\[
	\varphi_{\delta}(x) = 
	\begin{cases}
		1, & \text{if } x\in B_{\frac{\delta}{2}}(0), \\
		0, & \text{if } x\in \mathbb{R}^d \setminus B_{\delta}(0),
	\end{cases}
	\]
	$0\leq \varphi_{\delta} \leq 1 $    in $\mathbb{R}^d$,
	and 
	\[
 |\nabla \varphi_{\delta}| \leq K \varphi_{\delta}^{\frac{1}{2}} \ \text{ in } \ \mathbb{R}^d,
	\]
	where $K$ is a positive constant of order $\mathcal{O}(\delta^{-1})$.
\end{lemma}

In particular, for any $q \in \overline{\Omega}$, the translated cut-off function $\varphi_{\delta}(x-q)$ satisfies the same properties. 

Next, we prepare a uniform-in-time lower bound for $s_{1}+c$.

\begin{lemma}
Assume that either $s_{1}>0$, or $s_{1}\ge0$ and the system is fluid-free. Then, there exists $C=C(\Omega, s_1, \int_{\Omega}n_0)>0$ such that 
 \begin{equation}\label{CLOWER}
\inf_{\overline{\Omega}}\{s_{1}+c(\cdot,t)\}\ge C\quad\mbox{for all }t<T_{\rm max}.
 \end{equation}
\end{lemma}

%\begin{proof}
%In the fluid-free case, we note from Lemma~\ref{ULE} and $\int_{\Omega}n=\int_{\Omega}n_{0}$ that there exists  $C>0$ satisfying $c(x) \ge C\int_{\Omega}n_{0}$ for $x\in \Omega$. Thus, in any case whether $s_{1}>0$, or $s_{1}\ge0$ and the system is fluid-free, we have \eqref{CLOWER}.
%\end{proof}

\begin{proof}
If $s_1>0$, then $s_1+c(\cdot,t)\ge s_1$ in $\Omega$ for all $t<T_{\max}$, and hence \eqref{CLOWER} holds with $C=s_1$.

In the fluid-free case, $\eqref{KS01}_2$ reduces to $-\Delta c + c = n$ in $\Omega$ with $\partial_\nu c=0$ on $\partial\Omega$.
By Lemma~\ref{ULE} and $\int_\Omega n(\cdot,t)=\int_\Omega n_0>0$, we obtain
$c(\cdot,t)\ge C\int_\Omega n_0$ in $\Omega$ for all $t<T_{\max}$.
This yields \eqref{CLOWER} when $s_1\ge0$.
\end{proof}

Next, we prepare some uniform-in-time estimates for $1/(s_{1}+c)$ and
\begin{equation}\label{FunctionF}
 F(c):=  \int_{1}^{c}\frac{1}{(s_{1}+s)^{k}}ds \quad  \mbox{for } k>\frac{1}{2}.
\end{equation}

\begin{lemma}
Assume that either $s_{1}>0$, or $s_{1}\ge0$ and the system is fluid-free.
For $k>\frac{1}{2}$ and $p\in[1,\infty)$, there exists $C=C(\Omega, k, p, s_1, \int_{\Omega} n_0)>0$ such that 
\begin{equation}\label{FLEM}
 \|F(c)(\cdot,t)\|_{L^{p}(\Omega)}+\|\nabla F(c)(\cdot,t)\|_{L^{2}(\Omega)} \le C\quad\mbox{for all }t<T_{\rm max},
\end{equation}
where $F$ is defined by \eqref{FunctionF}.
\end{lemma}
\begin{proof}
A direct computation shows that
\[
F(c)=
\begin{cases}
\frac{1}{1-k} \left[(s_1+c)^{1-k}-(s_1+1)^{1-k} \right], & \text{if } k\neq 1,\\[1mm]
\ln(s_1+c)-\ln(s_1+1), & \text{if } k=1,
\end{cases}
\qquad\text{and}\qquad
\nabla F(c)=\frac{\nabla c}{(s_1+c)^k}.
\]

Choose $(\sigma_1,r)=(s_1,2k-1)$ in \eqref{NABLACL2}. Since $k>\frac{1}{2}$, by applying \eqref{CLOWER}, it follows that  
\[
(2k-1)\int_\Omega \frac{|\nabla c|^2}{(s_1+c)^{2k}}
\le \int_\Omega \frac{c}{(s_1+c)^{2k-1}} \le C^{1-2k} \int_{\Omega} c = C^{1-2k} \int_{\Omega} n_0.
\]
This yields the desired uniform bound for $\|\nabla F(c)\|_{L^2(\Omega)}$.

Next, in the view of the Gagliardo-Nirenberg inequality, it is enough to show the uniform-in-time bound of $\| F(c)(\cdot,t)\|_{L^{1}(\Omega)}$.
For $\|F(c)\|_{L^1(\Omega)}$, we consider three cases. \emph{(i)} If $\frac{1}{2}<k<1$, then $1-k \in (0,1)$ and $(s_1 + c)^{1-k} \leq C(1+c)$ for all $c \geq 0$ using $a^{1-k} \le 1+a$ for all $a\ge 0$. Because $\int_{\Omega} c = \int_{\Omega} n_0$, it follows that
\[
\|F(c)\|_{L^1(\Omega)}\le C\left(1+\|c\|_{L^1(\Omega)}\right)
= C\left(1+\|n_0\|_{L^1(\Omega)}\right).
\]
\emph{(ii)} If $k>1$, then $1-k<0$ and $(s_1+c)^{1-k}\le C^{1-k}$ by \eqref{CLOWER}. Hence $\|F(c)\|_{L^1(\Omega)}\le C$.  
\emph{(iii)} If $k=1$, then $\ln (s_1 +c)$ is bounded from below using \eqref{CLOWER}, while from above using $\ln a \le a$ for $a>0$. Therefore, $\|F(c)\|_{L^1(\Omega)}\le C\left(1+\|c\|_{L^1(\Omega)}\right)\le C$.
\end{proof}

%Next, we derive some local estimates using the cut-off function $\varphi_{\delta}$. 

Next, we derive localized estimates by means of the cut-off function $\varphi_\delta$ from Lemma~\ref{LEMSTEST}.
In particular, we obtain a local $L^2$-bound for $n$ and a local $L^4$-bound for $\nabla c/(s_1+c)^k$, which will be used in the subsequent energy arguments.

\begin{lemma}\label{LEMLOCINT}
Let $k>\frac{1}{2}$.
Assume that either $s_{1}>0$, or $s_{1}\ge0$ and the system is fluid-free.
Denote $\varphi(x)=\varphi_{\delta}(x-q)$ and $B_{\delta}=B_{\delta}(q)$  for $q\in\overline{\Omega}$, where $\varphi_{\delta}$, $\delta>0$,  is the function given in   Lemma~\ref{LEMSTEST}. 
Then, there exist two positive constants $C_{1}=C_1(\Omega, \int_{\Omega} n_0) >0$ and  $C_2=C_2(\Omega, k, s_1, \int_{\Omega} n_0)>0$ independent of $\delta$ and $q$  such that
\begin{equation}\label{LEMLOCINT1}
\int_{\Omega}n^{2}\varphi^{4}
 \le C_{1} \bke{  \int_{\Omega} \frac{|\nabla n|^{2}}{n}\varphi^{4}  +\| \varphi^{2}  \|_{W^{1,\infty}(\R^{2})}^{2}},
\end{equation}
\begin{align}\label{LEMLOCINT2}
\begin{aligned} 
 \int_{\Omega} \frac{|\nabla c|^{4}}{(s_{1}+c)^{4k}}\varphi^{4}
 &\le  C_{2}\norm{\frac{\nabla c}{(s_{1}+c)^{k}}}_{L^{2}(\Omega\cap B_{\delta})} \norm{\frac{\nabla c}{(s_{1}+c)^{\frac{3-k}{2}}}}_{L^{2}(\Omega\cap B_{\delta})}^{2}  \int_{\Omega} \frac{|\nabla c|^{4}}{(s_{1}+c)^{4k}}\varphi^{4}
   \\
 &\quad +  C_{2}\norm{\frac{\nabla c}{(s_{1}+c)^{k}}}_{L^{2}(\Omega\cap B_{\delta})}
\bke{   \int_{\Omega} \frac{|\nabla n|^{2}}{n}\varphi^{4}  +\| \varphi^{2}  \|_{W^{1,\infty}(\R^{2})}^{2}  +\|\varphi^{\frac{4}{3}}\|_{W^{2,2}(\R^{2})}^{3}} \\
&\quad +C_{2} \norm{\frac{\nabla c}{(s_{1}+c)^{k}}}^{4}_{L^{2}(\Omega\cap B_{\delta})} \bke{\int_{\Omega}|u|^{\frac{12}{5}} }^{5}
\end{aligned}
\end{align}
\end{lemma}

\begin{proof}
First, to prove \eqref{LEMLOCINT1}, we use the Sobolev embedding $W^{1, 1}(\Omega)\hookrightarrow L^{2}(\Omega)$ valid in two dimensions that there exists $C=C(\Omega)>0$ satisfying
\[
 \int_{\Omega}n^{2}\varphi^{4}   \le C\bke{ \|\nabla(n\varphi^{2})\|_{L^{1}(\Omega)}^{2}+\|n\varphi^{2}\|_{L^{1}(\Omega)}^{2}}.
\]
Using   H\"older's inequality and  $\int_{\Omega}n=\int_{\Omega}n_{0}$, there is $C>0$, independent of $\delta$ and $q$, such that 
\[
\begin{aligned}
 \|\nabla(n\varphi^{2})\|_{L^{1}(\Omega)}^{2} +\|n\varphi^{2}\|_{L^{1}(\Omega)}^{2}
\le C\bke{ \|n_{0}\|_{L^{1}(\Omega)}  \int_{\Omega} \frac{|\nabla n|^{2}}{n}\varphi^{4} +\|n_{0}\|_{L^{1}(\Omega)}^{2}\| \varphi^{2}  \|_{W^{1,\infty}(\R^{2})}^{2} }.
 \end{aligned}
\]
We prove \eqref{LEMLOCINT1} by combining two inequalities above.

Recall $F$ from \eqref{FunctionF}, so that $\nabla F(c)=\frac{\nabla c}{(s_1+c)^k}$.
Now, using  H\"older's  inequality and $(a_{1}+a_{2})^{3}\le 4(a_{1}^{3}+a_{2}^{3})$ for $a_{1},a_{2}\ge0$,  we compute the following 
\begin{equation}\label{LEM5_1}
\begin{aligned}[b]
 \int_{\Omega}\frac{|\nabla c|^{4}}{(s_{1}+c)^{4k}}\varphi^{4}
&\le   \norm{\frac{\nabla c}{(s_{1}+c)^{k}}}_{L^{2}(\Omega\cap B_{\delta})}\bke{\int_{\Omega}\frac{|\nabla c|^{6}}{(s_{1}+c)^{6k}}\varphi^{8}}^{\frac{1}{2}}\\
%&=  \norm{\frac{\nabla c}{(s_{1}+c)^{k}}}_{L^{2}(\Omega\cap B_{\delta})} \norm{ \nabla F(c) \varphi^{\frac{4}{3}} }_{L^{6}(\Omega)}^{3}\\
 &= \norm{\frac{\nabla c}{(s_{1}+c)^{k}}}_{L^{2}(\Omega\cap B_{\delta})}\norm{ \nabla (F(c)\varphi^{\frac{4}{3}})- F(c) \nabla \varphi^{\frac{4}{3}} }_{L^{6}(\Omega)}^{3}
\\
 &\le   \norm{\frac{\nabla c}{(s_{1}+c)^{k}}}_{L^{2}(\Omega\cap B_{\delta})}\bke{\norm{  F(c)  \varphi^{\frac{4}{3}}}_{W^{1,6}(\Omega)} + \norm{F(c)\nabla \varphi^{\frac{4}{3}}}_{L^{6}(\Omega)}}^{3}
\\
& \le  4\norm{\frac{\nabla c}{(s_{1}+c)^{k}}}_{L^{2}(\Omega\cap B_{\delta})} \bke{\norm{  F(c)  \varphi^{\frac{4}{3}}}_{W^{1,6}(\Omega)}^{3}+  \norm{ F(c)}_{L^{12}(\Omega)}^{3}\norm{ \nabla \varphi^{\frac{4}{3}}}_{L^{12}(\Omega)}^{3}}.
\end{aligned}
\end{equation}
Note that $\norm{ F(c)}_{L^{12}(\Omega)}$ is uniformly bounded by Lemma~\ref{FLEM} with $p=12$.

First, we estimate $\norm{   F(c)  \varphi^{\frac{4}{3}}}_{W^{1,6}(\Omega)}$.
The Sobolev embedding $W^{2,\frac{3}{2}}(\Omega)\hookrightarrow W^{1,6}(\Omega)$ along with 
%the standard elliptic regularity theory 
a standard $W^{2,p}$-estimate with $p=\frac{3}{2}$
yields that there exists $C=C(\Omega)>0$ independent of $\delta$ and $q$ satisfying
\begin{equation}\label{LEM5_2}
\begin{aligned}
 \norm{   F(c)  \varphi^{\frac{4}{3}}}_{W^{1,6}(\Omega)}^{3}
 \le C\bke{  \norm{\Delta \bke{  F(c)   \varphi^{\frac{4}{3}} } }_{L^{\frac{3}{2}}(\Omega)}^{3} +\norm{    F(c)   \varphi^{\frac{4}{3}}}_{L^{\frac{3}{2}}(\Omega)}^{3} +\norm{ F(c)\nabla \varphi^{\frac{4}{3}} \cdot \nu }_{W^{\frac{1}{3},\frac{3}{2}}(\partial\Omega)}^{3} },
\end{aligned}
\end{equation}
where we used
$\nabla(F(c)\varphi^{\frac{4}{3}})\cdot\nu=F(c)\nabla \varphi^{\frac{4}{3}} \cdot \nu$ on $\partial\Omega$ since $\partial_{\nu} c = 0$ on $\partial \Omega$. 
By H\"older's inequality and   $W^{1,2}(\Omega)\hookrightarrow L^{\frac{12}{7}}(\Omega)$, we can find $C=C(\Omega)>0$ such that
\[
\| F(c)   \varphi^{\frac{4}{3}}\|_{L^{\frac{3}{2}}(\Omega)}^{3}\le C \norm{ F(c)}_{L^{12}(\Omega)}^{3}\|\varphi^{\frac{4}{3}}\|_{W^{2,2}(\R^{2})}^{3}.
\]
Note   from the trace inequality and a direct computation that
\begin{align*}
\begin{aligned}
\| F(c)\nabla \varphi^{\frac{4}{3}} \cdot \nu \|_{W^{\frac{1}{3},\frac{3}{2}}(\partial\Omega)}^{3}\le C \|F(c)\nabla \varphi^{\frac{4}{3}} \|_{W^{1,\frac{3}{2}}( \Omega)}^{3}.
\end{aligned}
\end{align*}
Note also from H\"older's inequality and $W^{1,2}(\Omega)\hookrightarrow L^{q}(\Omega)$ for all $q\in[1,\infty)$ that with some $C=C(\Omega)>0$
\[
\|F(c)\nabla \varphi^{\frac{4}{3}} \|_{W^{1,\frac{3}{2}}( \Omega)}^{3}\le C(\|F(c)\|_{L^{12}(\Omega)}^{3}+\|\nabla F(c)\|_{L^{2}(\Omega)}^{3}+ \|F(c)\|_{L^{6}(\Omega)}^{3})\|\varphi^{\frac{4}{3}}\|_{W^{2,2}(\R^{2})}^{3}.
\]
By applying Lemma~\ref{FLEM} with $p=12$ and $p=6$, the last two terms of \eqref{LEM5_2} are bounded uniformly. 
Hence it remains to estimate $\|\Delta(F(c)\varphi^{4/3})\|_{L^{3/2}}^{3}$.
Using  H\"older's inequality,  $(a_{1}+a_{2})^{3}\le 4(a_{1}^{3}+a_{2}^{3})$ for $a_{1},a_{2}\ge0$,  $\eqref{KS01}_{2}$, and $W^{2,2}(\Omega)\hookrightarrow W^{1,6}(\Omega)$, we compute 
\begin{align*}
\begin{aligned}
\|  \Delta(F(c)   \varphi^{\frac{4}{3}})  \|_{L^{\frac{3}{2}}(\Omega)}^{3}
&\le 4\| \Delta F(c)   \varphi^{\frac{4}{3}}   \|_{L^{\frac{3}{2}}(\Omega)}^{3}
    +4 \left( 2\| \nabla F(c) \cdot \nabla \varphi^{\frac{4}{3}}\|_{L^{\frac{3}{2}}(\Omega)}+\| F( c)  \Delta \varphi^{\frac{4}{3}}   \|_{L^{\frac{3}{2}}(\Omega)}\right)^{3}
\\
&\le  16\bke{\norm{\frac{ u\cdot \nabla c-n+c  }{(s_{1}+c)^{k}}  \varphi^{\frac{4}{3}}   }_{L^{\frac{3}{2}}(\Omega)}^{3}
+  k^{3}\norm{ \frac{|\nabla c|^{2}}{(s_{1}+c)^{k+1}}      \varphi^{\frac{4}{3}}   }_{L^{\frac{3}{2}}(\Omega)}^{3}} 
\\&\quad+ 8\left( \|\nabla F(c)  \|_{L^{2}(\Omega )} +\|F(c)\|_{L^{6}(\Omega)} \right)^{3} \|\varphi^{\frac{4}{3}}\|^{3}_{W^{2,2}(\R^{2})}  .
\end{aligned}
\end{align*}
The last term is bounded uniformly by Lemma~\ref{FLEM} with $p=6$.
By applying H\"{o}lder inequality, we note that 
\begin{align*}
\begin{aligned}
\norm{ \frac{|\nabla c|^{2}}{(s_{1}+c)^{k+1}}      \varphi^{\frac{4}{3}}   }_{L^{\frac{3}{2}}(\Omega)}^{3}  \le \norm{\frac{\nabla c}{(s_{1}+c)^{\frac{3-k}{2}}}}_{L^{2}(\Omega\cap B_{\delta})}^{2}  \int_{\Omega} \frac{|\nabla c|^{4}}{(s_{1}+c)^{4k}}\varphi^{4}. 
\end{aligned}
\end{align*}
Now, it follows that, by the H\"{o}lder inequality along with $ \varphi \le 1$, 
\begin{align*}
\begin{aligned}
\norm{\frac{ u\cdot \nabla c  }{(s_{1}+c)^{k}}  \varphi^{\frac{4}{3}}   }_{L^{\frac{3}{2}}(\Omega)}^{3}
\le \bke{\int_{\Omega}|u|^{\frac{12}{5}} }^{\frac{5}{4}}\bke{\int_{\Omega} \frac{|\nabla c|^{4}}{(s_{1}+c)^{4k}}\varphi^{4}}^{\frac{3}{4}}.
\end{aligned}
\end{align*}
There is a constant $C=C(\Omega, s_1, k, \int_{\Omega}n_0)>0$, by H\"older's inequality along with \eqref{NABLACL2} and \eqref{CLOWER}, such that 
\begin{align*}
\begin{aligned}
\norm{ \frac{ n }{(s_{1}+c)^{k}}  \varphi^{\frac{4}{3}}   }_{L^{\frac{3}{2}}(\Omega)}^{3} 
&\le  \norm{\frac{n}{(s_{1}+c)^{2k}}}_{L^{1}( \Omega \cap B_{\delta} )}\int_{\Omega}n^{2}\varphi^{4} 
\\&  \le \left(\int_{\Omega} \frac{1}{(s_{1}+c)^{2k-1}} \right) \int_{\Omega}n^{2}\varphi^{4} 
\\&\le  C\int_{\Omega}n^{2}\varphi^{4},
\end{aligned}
\end{align*}
Similarly, by H\"older's inequality, \eqref{CLOWER}, $\int_{\Omega}c=\int_{\Omega}n_{0}$, and $W^{2,2}(\Omega)\hookrightarrow L^{6}(\Omega)$,
there is $C>0$ such that % {\color{dred} $C=C(\Omega, s_1, \int_{\Omega} n_0)>0$}
\begin{align*}
\begin{aligned}
\norm{ \frac{ c }{(s_{1}+c)^{k}}  \varphi^{\frac{4}{3}}   }_{L^{\frac{3}{2}}(\Omega)}^{3} 
&\le \norm{ \frac{ c }{(s_{1}+c)^{k}}}_{L^{2}(\Omega \cap B_{\delta} )}^{3} \norm{  \varphi^{\frac{4}{3}}   }_{L^{6}(\Omega)}^{3} 
\\&  \le \left(\int_{\Omega} \frac{c}{(s_1 + c)^{2k-1}}\right)^{\frac{3}{2}} \norm{  \varphi^{\frac{4}{3}}   }_{L^{6}(\Omega)}^{3} 
\\&\le C \norm{\varphi^{\frac{4}{3}}}_{W^{2,2}(\R^{2})}^{3} .
\end{aligned}
\end{align*}
Therefore, there is a constant $C>0$ such that 
\[\begin{aligned}
\| \Delta(F(c)   \varphi^{\frac{4}{3}})  \|_{L^{\frac{3}{2}}(\Omega)}^{3}
&\leq C\int_{\Omega}n^{2}\varphi^{4} 
+ 16 \bke{\int_{\Omega}|u|^{\frac{12}{5}} }^{\frac{5}{4}}\bke{\int_{\Omega} \frac{|\nabla c|^{4}}{(s_{1}+c)^{4k}}\varphi^{4}}^{\frac{3}{4}} \\
&\quad + 16 k^3 \norm{\frac{\nabla c}{(s_{1}+c)^{\frac{3-k}{2}}}}_{L^{2}(\Omega\cap B_{\delta})}^{2}  \int_{\Omega} \frac{|\nabla c|^{4}}{(s_{1}+c)^{4k}}\varphi^{4}
 + C \norm{\varphi^{\frac{4}{3}}}_{W^{2,2}(\R^{2})}^{3}.
\end{aligned}\]
Moreover, there is a constant $C>0$ such that 
\[\begin{aligned}
\norm{   F(c)  \varphi^{\frac{4}{3}}}_{W^{1,6}(\Omega)}^{3}
&\leq C\int_{\Omega}n^{2}\varphi^{4} 
+ C\bke{\int_{\Omega}|u|^{\frac{12}{5}} }^{\frac{5}{4}}\bke{\int_{\Omega} \frac{|\nabla c|^{4}}{(s_{1}+c)^{4k}}\varphi^{4}}^{\frac{3}{4}} \\
&\quad + C k^3 \norm{\frac{\nabla c}{(s_{1}+c)^{\frac{3-k}{2}}}}_{L^{2}(\Omega\cap B_{\delta})}^{2}  \int_{\Omega} \frac{|\nabla c|^{4}}{(s_{1}+c)^{4k}}\varphi^{4}
 + C \norm{\varphi^{\frac{4}{3}}}_{W^{2,2}(\R^{2})}^{3}.
\end{aligned}\]

Combining the above estimates
shows that there exists a constant $C>0$ independent of $\delta$ and $q$ such that
\begin{align*}
\begin{aligned}
 \int_{\Omega}\frac{|\nabla c|^{4}}{(s_{1}+c)^{4k}}\varphi^{4} 
 & \le  C\norm{\frac{\nabla c}{(s_{1}+c)^{k}}}_{L^{2}(\Omega\cap B_{\delta})} \norm{\frac{\nabla c}{(s_{1}+c)^{\frac{3-k}{2}}}}_{L^{2}(\Omega\cap B_{\delta})}^{2}  \int_{\Omega} \frac{|\nabla c|^{4}}{(s_{1}+c)^{4k}}\varphi^{4}
 \\ & \quad +
  C \norm{\frac{\nabla c}{(s_{1}+c)^{k}}}_{L^{2}(\Omega\cap B_{\delta})}   \bkt{ \bke{\int_{\Omega} \frac{|\nabla c|^{4}}{(s_{1}+c)^{4k}}\varphi^{4}}^{\frac{3}{4}}\bke{\int_{\Omega}|u|^{\frac{12}{5}} }^{\frac{5}{4}} + \int_{\Omega}n^{2}\varphi^{4}  +\|\varphi^{\frac{4}{3}}\|_{W^{2,2}(\R^{2})}^{3}}.
\end{aligned}
\end{align*}
Finally, by Young's inequality, \eqref{NABLACL2} and   \eqref{LEMLOCINT1}, it follows that there exists a constant $C=C(\Omega, s_1, k, \int_{\Omega}n_0)>0$ independent of $\delta$ and $q$ satisfying
\[
\begin{aligned}
  \int_{\Omega}\frac{|\nabla c|^{4}}{(s_{1}+c)^{4k}}\varphi^{4} 
 &  \le C\norm{\frac{\nabla c}{(s_{1}+c)^{k}}}_{L^{2}(\Omega\cap B_{\delta})} \norm{\frac{\nabla c}{(s_{1}+c)^{\frac{3-k}{2}}}}_{L^{2}(\Omega\cap B_{\delta})}^{2}  \int_{\Omega} \frac{|\nabla c|^{4}}{(s_{1}+c)^{4k}}\varphi^{4}
 \\&\quad  +   C \norm{\frac{\nabla c}{(s_{1}+c)^{k}}}_{L^{2}(\Omega\cap B_{\delta})}     \bke{  \int_{\Omega} \frac{|\nabla n|^{2}}{n}\varphi^{4}  +\| \varphi^{2}  \|_{W^{1,\infty}(\R^{2})}^{2}  +\|\varphi^{\frac{4}{3}}\|_{W^{2,2}(\R^{2})}^{3}}
 \\&\quad + C \norm{\frac{\nabla c}{(s_{1}+c)^{k}}}^{4}_{L^{2}(\Omega\cap B_{\delta})} \bke{\int_{\Omega}|u|^{\frac{12}{5}} }^{5},
 \end{aligned}
\]
namely, \eqref{LEMLOCINT2}.
\end{proof}

%The spatially localized $L\ln L$-norm of $n$ is uniformly bounded.

Next, we prove a localized entropy estimate for $n$.
Testing \eqref{KS01}$_1$ by $\ln n\,\varphi^4$ and combining the resulting identity with Lemma~\ref{LEMLOCINT} and Proposition~\ref{KEYPRO}, we obtain uniform bounds for the localized $L\ln L$-norm of $n$.

\begin{lemma}\label{LEMLOGL}
Assume that either $\gamma>\frac{1}{2}$ and $s_{1}>0$, or  $\gamma>0$,  $s_{1}\ge0$ and the system is fluid-free.
Denote $\varphi(x)=\varphi_{\delta}(x-q)$ and $B_{\delta}=B_{\delta}(q)$  for $q\in\overline{\Omega}$, where $\varphi_{\delta}$, $\delta>0$,  is the function given in   Lemma~\ref{LEMSTEST}.  Then, there exist   $\delta_{*}>0$ independent of $q $ such that if $\delta \in (0, \delta_{*})$, then there exists 
   $C=C(\delta, \Omega, s_0, s_1, \gamma, \int_{\Omega}n_0, \|u_0\|_{L^2(\Omega)}, \allowbreak\|\nabla \Phi\|_{L^{\infty}(\Omega)})>0$ independent of $q $  satisfying
\[
\sup_{t<T_{\rm max}} \int_{\Omega}  (n\ln n(\cdot,t) +e^{-1})\varphi^{4}\le C,
\]
and 
\[
\int_{t}^{t+\tau}\int_{\Omega}\frac{|\nabla n|^{2}}{n}\varphi^{4} \le C\quad\mbox{ for all }t\in(0,T_{\rm max}-\tau),
\]
where $\tau:=\min\{1,\frac{1}{2}T_{\rm max}\}$.
\end{lemma}

\begin{proof}
For simplicity, let $\varphi=\varphi_\delta$ be given by Lemma~\ref{LEMSTEST}.
Throughout, $C$ denotes a generic constant, and we will explicitly track its dependence on $\delta$. Since $n>0$ and $\varphi$ is independent of $t$, we have
\[
n_t \ln n\,\varphi^4 = \partial_t(n\ln n-n)\,\varphi^4.
\]
Multiplying \eqref{KS01}$_1$ by $\ln n\,\varphi^{4}$ and integrating over $\Omega$, we obtain  
\begin{equation}\label{MAR66}
\begin{aligned}
 &\frac{d}{dt}\int_{\Omega}  n\ln n \varphi^{4}- \frac{d}{dt}\int_{\Omega}n\varphi^{4} + \int_{\Omega} \frac{|\nabla n|^{2}}{n}\varphi^{4} \\
&\quad  = -\int_{\Omega} \ln n\nabla n\cdot \nabla \varphi^{4}  -\int_{\Omega} n u \cdot \nabla \varphi^4 + \int_{\Omega} \left( n \ln n \right)u \cdot \nabla \varphi^4 \\
&\qquad + \int_{\Omega}\nabla n  \cdot ( S(x,n,c)\cdot\nabla c) \varphi^{4}
  + \int_{\Omega}n\ln n (S(x,n,c)\cdot\nabla c  ) \cdot\nabla \varphi^4  \\
&\quad  =: I_1 + I_2 + I_3 + I_4 + I_5.
\end{aligned}
\end{equation}
Here we used integration by parts with no-flux boundary condition $\left(\nabla n - n S(x,n,c) \cdot \nabla c\right)\cdot \nu$ on $\partial \Omega$
\[\begin{aligned}
&\int_{\Omega}\Delta n \,\ln n\,\varphi^{4} -\int_{\Omega}\nabla\cdot\bigl(nS(x,n,c) \cdot \nabla c\bigr)\,\ln n\,\varphi^{4} \\
%&\quad = \int_{\Omega} \left[- \nabla n  + n S(x,n,c) \cdot \nabla c \right]\cdot \nabla(\ln n \, \varphi^4)\\
&\quad =-\int_{\Omega}\frac{|\nabla n|^{2}}{n}\varphi^{4}
-\int_{\Omega}\ln n\,\nabla n\cdot\nabla\varphi^{4}
+\int_{\Omega} n\,\bigl(S(x,n,c)\nabla c\bigr)\cdot\nabla(\ln n\,\varphi^{4}),
\end{aligned}\]
integration by parts with $\nabla \cdot u =0 $ and $u=0$ on $\partial \Omega$
\[
-\int_{\Omega} u\cdot\nabla n\,\ln n\,\varphi^{4}
= -\int_{\Omega} u\cdot\nabla(n\ln n-n)\,\varphi^{4}
= \int_{\Omega}(n\ln n-n)\,u\cdot\nabla\varphi^{4}.
\]

\smallskip

Let us estimate $I_1$, $I_2$, and $I_3$ of \eqref{MAR66}, diffusion and fluid terms. By Young's inequality, we have 
\[
I_1 \le \frac{1}{8}\int_{\Omega} \frac{|\nabla n|^{2}}{n}\varphi^{4}+ 32\int_{\Omega} n|\ln n|^{2} \varphi^{2} |\nabla \varphi|^{2}. 
\]
Using $a|\ln a|^{2}\le 16e^{-2}a^{\frac{3}{2}}+4e^{-2}$ for $a\ge0$ and Lemma~\ref{LEMSTEST} (i.e.\ $|\nabla\varphi|\le K\varphi^{1/2}$),
\[\begin{aligned}
I_1 &\leq  \frac{1}{8}\int_{\Omega} \frac{|\nabla n|^{2}}{n}\varphi^{4}+ 32\int_{\Omega} (16e^{-2}n^{\frac{3}{2}}+4e^{-2}) \varphi^{2} |\nabla \varphi|^{2}
\\&
\le  \frac{1}{8}\int_{\Omega} \frac{|\nabla n|^{2}}{n}\varphi^{4}
+ C K^{2} \int_{\Omega}n^{\frac{3}{2}}\varphi^{3}+C K^{2}\int_{\R^{2}}\varphi^{3}.
\end{aligned}\]
Moreover, by Young's inequality and the mass conservation of $n$,
\begin{equation}\label{S2}
\left(\int_{\Omega}n^{\frac{3}{2}}\varphi^{2}\right)^{2}\le \int_{\Omega}n\int_{\Omega}n^{2}\varphi^{4}
=\|n_{0}\|_{L^{1}(\Omega)}\int_{\Omega}n^{2}\varphi^{4}. 
\end{equation}
Applying \eqref{LEMLOCINT1} and Young's inequality yields
\begin{equation}\label{FORTHTERM}
I_1\le\frac{1}{4}\int_{\Omega} \frac{|\nabla n|^{2}}{n}\varphi^{4}+C(\delta),
\end{equation}
because $K$ is depending on $\delta$ in Lemma~\ref{LEMSTEST}. 

Now, by H\"{o}lder's inequality, Lemma~\ref{UBDD}, \eqref{LEMLOCINT1}, and by Young's inequality, we obtain that 
\begin{align*}
\begin{aligned}
I_2
&\le 2\|\nabla \varphi^{2} \|_{L^{\infty}(\R^{2})}\|u\|_{L^{2}(\Omega)}\bke{\int_{\Omega}n^{2}\varphi^{4}}^{\frac{1}{2}}
\\& \le 2 C_{1}^{\frac{1}{2}} \| \varphi^{2}  \|_{W^{1,\infty}(\R^{2})}\|u\|_{L^{2}(\Omega)} \bke{  \int_{\Omega} \frac{|\nabla n|^{2}}{n}\varphi^{4}  +\| \varphi^{2}  \|_{W^{1,\infty}(\R^{2})}^{2}}^{\frac{1}{2}} 
\\& \le 4 C_1\| \varphi^{2}  \|_{W^{1,\infty}(\R^{2})}^{2}\|u\|_{L^{2}(\Omega)}^{2}   +\frac{1}{8}\bke{  \int_{\Omega} \frac{|\nabla n|^{2}}{n}\varphi^{4}  +\| \varphi^{2}  \|_{W^{1,\infty}(\R^{2})}^{2}}. 
\end{aligned}
\end{align*}
Next, we use H\"older's  inequality, $a^{\frac{12}{7}}|\ln a|^{\frac{12}{7}}\le 6^{\frac{12}{7}}e^{-\frac{12}{7}} a^{2}+e^{-\frac{12}{7}}$ for $a\ge0$, and Lemma~\ref{LEMSTEST} that 
\begin{align*}
\begin{aligned}
I_3
&\le 4\bke{\int_{\Omega} |u|^{\frac{12}{5}}}^{\frac{5}{12}}\bke{\int_{\Omega}n^{\frac{12}{7}}|\ln n|^{\frac{12}{7}} |\nabla \varphi|^{\frac{12}{7}} \varphi^{\frac{36}{7}}  }^{\frac{7}{12}}
\\ &\le 4\bke{\int_{\Omega} |u|^{\frac{12}{5}}}^{\frac{5}{12}}\bke{6^{\frac{12}{7}}e^{-\frac{12}{7}}K^{\frac{12}{7}}\int_{\Omega}n^{2}  \varphi^{6}  +e^{-\frac{12}{7}}K^{\frac{12}{7}}\int_{\Omega}   \varphi^{6}}^{\frac{7}{12}}
\\ &\le \frac{1}{16}\int_{\Omega} \frac{|\nabla n|^{2}}{n}\varphi^{4}+\bke{\int_{\Omega} |u|^{\frac{12}{5}}}^{5}+C(\delta),
\end{aligned}
\end{align*}
by applying \eqref{LEMLOCINT1} and Young's inequality for some constant $C(\delta)>0$. 
Therefore, we have 
\begin{equation}\label{I2+I3}
I_2 + I_3 \leq \frac{3}{16}\int_{\Omega} \frac{|\nabla n|^{2}}{n}\varphi^{4}+\bke{\int_{\Omega} |u|^{\frac{12}{5}}}^{5}+C(\delta).
\end{equation}

%To obtain an estimate for $I_1$ on \eqref{MAR66}, the Young's inequality yields that 
%\[
%I_1 \le \frac{1}{8}\int_{\Omega} \frac{|\nabla n|^{2}}{n}\varphi^{4}+ 32\int_{\Omega} n|\ln n|^{2} \varphi^{2} |\nabla \varphi|^{2}. 
%\]
%Then we apply the inequality, $a|\ln a|^{2}\le 16e^{-2}a^{\frac{3}{2}}+4e^{-2}$ for $a\ge0$, and Lemma~\ref{LEMSTEST}, that gives 
%\[\begin{aligned}
%I_1 &\leq  \frac{1}{8}\int_{\Omega} \frac{|\nabla n|^{2}}{n}\varphi^{4}+ 32\int_{\Omega} (16e^{-2}n^{\frac{3}{2}}+4e^{-2}) \varphi^{2} |\nabla \varphi|^{2}
%\\&
%\le  \frac{1}{8}\int_{\Omega} \frac{|\nabla n|^{2}}{n}\varphi^{4}+ 32\cdot 16e^{-2}K^{2} \int_{\Omega}n^{\frac{3}{2}}\varphi^{3}+32\cdot 4 e^{-2}K^{2}\int_{\Omega}\varphi^{3}.
%\end{aligned}\]
%  Since the following inequality by the Young's inequality and mass conservation property
%\begin{equation}\label{S2}
%\bke{\int_{\Omega}n^{\frac{3}{2}}\varphi^{2}}^{2}\le \int_{\Omega}n\int_{\Omega}n^{2}\varphi^{4}=\|n_{0}\|_{L^{1}(\Omega)}\int_{\Omega}n^{2}\varphi^{4}, 
%\end{equation}
%we have 
%\begin{equation*}
%\begin{aligned}
%I_1 \le  \frac{1}{8}\int_{\Omega} \frac{|\nabla n|^{2}}{n}\varphi^{4}+ 32\cdot 16 e^{-2}K^{2}  \|n_{0}\|_{L^{1}(\Omega)}^{\frac{1}{2}}\bke{\int_{\Omega}n^{2}\varphi^{4}}^{\frac{1}{2}}+32\cdot 4 e^{-2}K^{2}\int_{\R^{2}}\varphi^{3} 
%\end{aligned}
%\end{equation*}
%By applying \eqref{LEMLOCINT1} and Young's inequality, there is constant $C=C(\delta)>0$ such that 
%\begin{equation}\label{FORTHTERM}
%I_1\le\frac{1}{4}\int_{\Omega} \frac{|\nabla n|^{2}}{n}\varphi^{4}+C(\delta).
%\end{equation}

% Handling III 

Now we derive estimates for the chemotaxis terms \(I_4\) and \(I_5\) in \eqref{MAR66}, considering separately the cases \(\gamma>\tfrac12\) and \(s_1>0\), or \(\gamma>0\), \(s_1\ge0\), and the system is fluid-free.

$\bullet$ \emph{Case 1: $\gamma>\frac{1}{2}$ and $s_{1}>0$.}
First of all, from \eqref{KS02}, Young's inequality and \eqref{CLOWER}, we note that for any $\varepsilon>0$ there exist $k\in (\frac{1}{2},1)$, $r\in(1,2)$, and  $C=C(\varepsilon, k, s_0, s_1, \int_{\Omega}n_{0})>0$ such that
\begin{align}\label{SSPLIT}
\begin{aligned}
|S(x,n,c)|&\le \frac{s_{0}}{(s_{1}+c)^{\gamma}}
\\&\le \frac{\varepsilon}{(s_{1}+c)^{k}}+\frac{C}{(s_{1}+c)^{r}}.
\end{aligned}
\end{align}

Using H\"{o}lder inequality, \eqref{SSPLIT}, and Young's inequality, we obtain 
\begin{align}\label{TUEMAR18}
\begin{aligned}
I_4 &\le \bke{\int_{\Omega} \frac{|\nabla n|^{2}}{n}\varphi^{4}}^{\frac{1}{2}}
\bke{\int_{\Omega}n^{2}\varphi^{4}}^{\frac{1}{4}}\bkt{\varepsilon\bke{\int_{\Omega}\frac{|\nabla c|^{4}}{(s_{1}+c)^{4k}}\varphi^{4}}^{\frac{1}{4}}+C\bke{\int_{\Omega}\frac{|\nabla c|^{4}}{(s_{1}+c)^{4r}}\varphi^{4}}^{\frac{1}{4}}}
\\&\le\frac{1}{16} \int_{\Omega} \frac{|\nabla n|^{2}}{n}\varphi^{4} +
\frac{1}{16C_{1}}\int_{\Omega}n^{2}\varphi^{4} +   C_{3}\varepsilon^{4}\int_{\Omega}\frac{|\nabla c|^{4}}{(s_{1}+c)^{4k}}\varphi^{4}+C_{4}\int_{\Omega}\frac{|\nabla c|^{4}}{(s_{1}+c)^{4r}}\varphi^{4},
\end{aligned}
\end{align}
for positive constants \(C_1\) as in \eqref{LEMLOCINT1} and \(C_3\), both independent of \(\varepsilon\), and for a constant \(C_4=C_4(\varepsilon)\).

Note that, by \eqref{NABLACL2},
$\int_{\Omega}c=\int_{\Omega}n_{0}$, and \eqref{CLOWER}, there exists  $C=C(k, r, \Omega, s_1, \|n_0\|_{L^1(\Omega)})>0$   such that
\begin{equation}\label{FEB10}
\sup_{t<T_{\rm max}}\bke{\norm{\frac{\nabla c}{(s_{1}+c)^{k}}(\cdot,t)}_{L^{2}(\Omega)}+\norm{\frac{\nabla c}{(s_{1}+c)^{\frac{3-r}{2}}}(\cdot,t)}_{L^{2}(\Omega)}^{2}}\le C.
\end{equation}
Thus, by Proposition~\ref{KEYPRO} with $\beta= 2-k >1$ and $\beta=2r-1>1$, there exists $\delta_{1}>0$ independent of $q\in\overline{\Omega}$ such that, for $\delta\in (0, \delta_{1})$,
\[
 C_{2}(k)\norm{\frac{\nabla c}{(s_{1}+c)^{k}}}_{L^{2}(\Omega\cap B_{\delta})} \norm{\frac{\nabla c}{(s_{1}+c)^{\frac{3-k}{2}}}}_{L^{2}(\Omega\cap B_{\delta})}^{2}  +C_{2}(r)\norm{\frac{\nabla c}{(s_{1}+c)^{r}}}_{L^{2}(\Omega\cap B_{\delta})} \norm{\frac{\nabla c}{(s_{1}+c)^{\frac{3-r}{2}}}}_{L^{2}(\Omega\cap B_{\delta})}^{2}  \le \frac{1}{2}
\]
with $C_2 (k) >0$ and $C_2 (r) >0$ as in \eqref{LEMLOCINT2}.
Hence, applying \eqref{LEMLOCINT2} with exponent $k$ and $r$, we obtain 
\begin{align}\label{LEMLOCINT11}
\begin{aligned} 
 \int_{\Omega} &\frac{|\nabla c|^{4}}{(s_{1}+c)^{4k}}\varphi^{4}
   \le
   2C_{2}(k)\norm{\frac{\nabla c}{(s_{1}+c)^{k}}}_{L^{2}(\Omega\cap B_{\delta})}
\bke{   \int_{\Omega} \frac{|\nabla n|^{2}}{n}\varphi^{4}   + C(\delta) }
+2C_{2}(k)\bke{\int_{\Omega}|u|^{\frac{12}{5}} }^{5},
\end{aligned}
\end{align}
and
\begin{align}\label{LEMLOCINT12}
\begin{aligned} 
 \int_{\Omega} &\frac{|\nabla c|^{4}}{(s_{1}+c)^{4r}}\varphi^{4}
   \le
   2C_{2}(r)\norm{\frac{\nabla c}{(s_{1}+c)^{r}}}_{L^{2}(\Omega\cap B_{\delta})}
\bke{   \int_{\Omega} \frac{|\nabla n|^{2}}{n}\varphi^{4}   +C(\delta)}
+2C_{2}(r)\bke{\int_{\Omega}|u|^{\frac{12}{5}} }^{5},
\end{aligned}
\end{align}
for all $\delta \in (0, \delta_1)$ and for a positive constant $C(\delta)= \| \varphi^{2}  \|_{W^{1,\infty}(\R^{2})}^{2}  +\|\varphi^{\frac{4}{3}}\|_{W^{2,2}(\R^{2})}^{3}$ by Lemma~\ref{LEMSTEST}. 
%\begin{align}\label{LEMLOCINT11}
%\begin{aligned} 
% \int_{\Omega} &\frac{|\nabla c|^{4}}{(s_{1}+c)^{4k}}\varphi^{4}
%   \\&\le
%   2C_{2}(k)\norm{\frac{\nabla c}{(s_{1}+c)^{k}}}_{L^{2}(\Omega\cap B_{\delta})}
%\bke{   \int_{\Omega} \frac{|\nabla n|^{2}}{n}\varphi^{4}   +\| \varphi^{2}  \|_{W^{1,\infty}(\R^{2})}^{2}  +\|\varphi^{\frac{4}{3}}\|_{W^{2,2}(\R^{2})}^{3}}
%+2C_{2}(k)\bke{\int_{\Omega}|u|^{\frac{12}{5}} }^{5},
%\end{aligned}
%\end{align}
%\begin{align}\label{LEMLOCINT12}
%\begin{aligned} 
% \int_{\Omega} &\frac{|\nabla c|^{4}}{(s_{1}+c)^{4r}}\varphi^{4}
%   \\&\le
%   2C_{2}(r)\norm{\frac{\nabla c}{(s_{1}+c)^{r}}}_{L^{2}(\Omega\cap B_{\delta})}
%\bke{   \int_{\Omega} \frac{|\nabla n|^{2}}{n}\varphi^{4}   +\| \varphi^{2}  \|_{W^{1,\infty}(\R^{2})}^{2}  +\|\varphi^{\frac{4}{3}}\|_{W^{2,2}(\R^{2})}^{3}}
%+2C_{2}(r)\bke{\int_{\Omega}|u|^{\frac{12}{5}} }^{5}.
%\end{aligned}
%\end{align}

%Here, by \eqref{NABLACL2}, $\int_{\Omega}c=\int_{\Omega}n_{0}$, and \eqref{CLOWER}, we note that there is a constant $C>0$ such that 
%\[
%\sup_{t<T_{\rm max}}\norm{\frac{\nabla c}{(s_{1}+c)^{k}}(\cdot,t)}_{L^{2}(\Omega)}\le C.
%\]
 
In view of \eqref{FEB10} and the third term on RHS of \eqref{TUEMAR18}, we can pick a sufficiently small $\varepsilon >0$ satisfying 
\begin{equation}\label{CONTROLK}
 2 C_{2}(k)C_{3}\varepsilon^{4}\sup_{t<T_{\rm max}}\norm{\frac{\nabla c}{(s_{1}+c)^{k}}(\cdot,t)}_{L^{2}(\Omega)}\le\frac{1}{32}.
\end{equation}
With such $\varepsilon$, due to Proposition~\ref{KEYPRO}, we can find $\delta_{*}\in (0, \delta_1)$independent of $q\in\overline{\Omega}$ such that 
\begin{equation}\label{DELSTAR}
\left( C_{2}(r)C_{4}+ \frac{1}{4}C_{2}(r) \right)\sup_{t<T_{\rm max}} \norm{\frac{\nabla c}{(s_{1}+c)^{r}}}_{L^{2}(\Omega\cap B_{\delta})} \le \frac{1}{16} 
\quad\mbox{for all } \delta \in (0, \delta_{*}).
\end{equation}
For any fixed $\delta \in (0, \delta_{*})$, then it follows from \eqref{TUEMAR18}, Lemma~\ref{LEMLOCINT}, and \eqref{CONTROLK} that  
\begin{align*}
\begin{aligned}
I_4 
%&\le   \frac{1}{8} \int_{\Omega} \frac{|\nabla n|^{2}}{n}\varphi^{4} 
%+C_{2}(k)C_{3}\varepsilon^{4}\norm{\frac{\nabla c}{(s_{1}+c)^{k}}}_{L^{2}(\Omega\cap B_{\delta})}\int_{\Omega} \frac{|\nabla n|^{2}}{n}\varphi^{4} 
%\\&\quad +  C_{2}(r)C_{4}\norm{\frac{\nabla c}{(s_{1}+c)^{r}}}_{L^{2}(\Omega\cap B_{\delta})}
% \int_{\Omega} \frac{|\nabla n|^{2}}{n}\varphi^{4} + \left(C_{2}(k)C_{3}\varepsilon^{4}+C_{2}(r)C_{4}\right)\bke{\int_{\Omega}|u|^{\frac{12}{5}} }^{5}+C(\delta)
 \le  \left(\frac{5}{32} +C_{2}(r)C_{4}\norm{\frac{\nabla c}{(s_{1}+c)^{r}}}_{L^{2}(\Omega\cap B_{\delta})} \right)
 \int_{\Omega} \frac{|\nabla n|^{2}}{n}\varphi^{4} 
 +\left(C_{2}(k)C_{3}\varepsilon^{4}+C_{2}(r)C_{4}\right)\bke{\int_{\Omega}|u|^{\frac{12}{5}} }^{5}+ C(\delta).
\end{aligned}
\end{align*}

Moreover, by \eqref{SSPLIT}, Young's inequality, $a^{\frac{4}{3}}|\ln a|^{\frac{4}{3}}\le 16e^{-\frac{4}{3}} a^{\frac{3}{2}}+e^{-\frac{4}{3}}$ for $a\ge0$, Lemma~\ref{LEMSTEST}, we have
\begin{align*}
\begin{aligned}
I_5 
%&= \left| \int_{\Omega}n\ln n  ( S(x,n,c)  \cdot \nabla c  )  \cdot\nabla \varphi^{4}  \right|
&\le  4\varepsilon\int_{\Omega} n|\ln n| \frac{|\nabla c|}{(s_{1}+c)^{k}}|\nabla \varphi| \varphi^{3}+ 4C\int_{\Omega} n|\ln n| \frac{|\nabla c|}{(s_{1}+c)^{r}}|\nabla \varphi| \varphi^{3}
\\&
\le C_{3}\varepsilon^{4}\int_{\Omega}\frac{|\nabla c|^{4}}{(s_{1}+c)^{4k}}\varphi^{4}+\frac{1}{4}\int_{\Omega}\frac{|\nabla c|^{4}}{(s_{1}+c)^{4r}}\varphi^{4}+C_{5}\int_{\Omega} n^{\frac{4}{3}}|\ln n|^{\frac{4}{3}}\varphi^{\frac{8}{3}}|\nabla \varphi|^{\frac{4}{3}}
\\&
\le C_{3}\varepsilon^{4}\int_{\Omega}\frac{|\nabla c|^{4}}{(s_{1}+c)^{4k}}\varphi^{4}+\frac{1}{4}\int_{\Omega}\frac{|\nabla c|^{4}}{(s_{1}+c)^{4r}}\varphi^{4}+C_{5}K^{\frac{4}{3}}\int_{\Omega} (16e^{-\frac{4}{3}} n^{\frac{3}{2}}+e^{-\frac{4}{3}})\varphi^{\frac{10}{3}} 
\\&
\le C_{3}\varepsilon^{4}\int_{\Omega}\frac{|\nabla c|^{4}}{(s_{1}+c)^{4k}}\varphi^{4}
+\frac{1}{4}\int_{\Omega}\frac{|\nabla c|^{4}}{(s_{1}+c)^{4r}}\varphi^{4}
+16C_{5}K^{\frac{4}{3}} e^{-\frac{4}{3}}\|n_{0}\|_{L^{1}(\Omega)}^{\frac{1}{2}}\bke{\int_{\Omega}n^{2}\varphi^{4}}^{\frac{1}{2}}
+C(\delta)
%+C_{5}K^{\frac{4}{3}}e^{-\frac{4}{3}} \int_{\R^{2}} \varphi^{\frac{10}{3}},
\end{aligned}
\end{align*}
for a constant $C_{5}(\varepsilon)>0$.
By using Lemma~\ref{LEMLOCINT}, \eqref{CONTROLK} and Young's inequality, we have
\begin{equation}
\begin{aligned}
I_5
\le  \left(\frac{1}{16} +\frac{1}{4}C_{2}(r)  \norm{\frac{\nabla c}{(s_{1}+c)^{r}}}_{L^{2}(\Omega\cap B_{\delta})} \right)\int_{\Omega} \frac{|\nabla n|^{2}}{n}\varphi^{4}+\left(C_{2}(k)C_{3}\varepsilon^{4}+\frac{1}{4}C_{2}(r)\right)\bke{\int_{\Omega}|u|^{\frac{12}{5}} }^{5}
+C(\delta).
\end{aligned}
\end{equation}
Therefore, by \eqref{DELSTAR}, we have 
\begin{equation}\label{I4+I5}
I_4 + I_5 \leq \frac{9}{32}\int_{\Omega} \frac{|\nabla n|^{2}}{n}\varphi^{4} +\left(2C_{2}(k)C_{3}\varepsilon^{4}+C_{2}(r)C_{4}+\frac{1}{4}C_{2}(r)\right)\bke{\int_{\Omega}|u|^{\frac{12}{5}} }^{5}
+C(\delta).
\end{equation}

Now the combination of estimates for all $I_i$, $i=1, \ldots, 5$ in \eqref{FORTHTERM}, \eqref{I2+I3}, and \eqref{I4+I5} to \eqref{MAR66} with \eqref{DELSTAR} yields the following, for some  $C(\delta)>0$, 
\begin{align}\label{LEM6PF1}
\begin{aligned} 
\frac{d}{dt} \int_{\Omega}  n\ln n\varphi^{4}-&\frac{d}{dt}\int_{\Omega}n\varphi^{4} +\frac{5}{32} \int_{\Omega} \frac{|\nabla n|^{2}}{n}\varphi^{4}
 \\&\le   \left( 2C_{2}(k)C_{3}\varepsilon^{4}+C_{2}(r)C_{4}+\frac{1}{4}C_{2}(r)+1\right)\bke{\int_{\Omega}|u|^{\frac{12}{5}} }^{5}+C(\delta). 
\end{aligned}
\end{align}

Let us define
\begin{equation}\label{FunctionG}
\mathcal{G}(t):=\int_{\Omega}  (n\ln n+e^{-1}) \varphi^{4}-\int_{\Omega}n\varphi^{4}.
\end{equation}
Note that $\int_\Omega n\varphi^4\le \int_\Omega n_0$ by the mass conservation of $n$ and $0\le\varphi\le 1$.
Moreover, by using $a\ln a\le 2e^{-1}a^{\frac{3}{2}}$ for $a\ge0$, $ \varphi  \le 1$, \eqref{S2}, \eqref{LEMLOCINT1}, and Young's inequality, there is $C(\delta)>0$  such that
\begin{align}\label{LEM6PF2}
\begin{aligned}
\int_{\Omega}  (n\ln n+e^{-1}) \varphi^{4}-\int_{\Omega}n\varphi^{4} &\le \int_{\Omega} 2e^{-1}n^{\frac{3}{2}} \varphi^{4}+e^{-1}\int_{\Omega}\varphi^{4}
\\&\le 2e^{-1} \|n_{0}\|_{L^{1}(\Omega)}^{\frac{1}{2}}\bke{\int_{\Omega}n^{2}\varphi^{4}}^{\frac{1}{2}} +e^{-1}\int_{\Omega}\varphi^{4}
\\&
\le  \frac{1}{8}\int_{\Omega} \frac{|\nabla n|^{2}}{n}\varphi^{4}+C(\delta).
\end{aligned}
\end{align}
By adding \eqref{LEM6PF2} to \eqref{LEM6PF1}, it follows that 
\[
\frac{d}{dt}\mathcal{G}(t)+\mathcal{G}(t) +\frac{1}{32}\int_{\Omega} \frac{|\nabla n|^{2}}{n}\varphi^{4}\le \left( 2C_{2}(k)C_{3}\varepsilon^{4}+C_{2}(r)C_{4}+\frac{1}{4}C_{2}(r)+1 \right) \bke{\int_{\Omega}|u|^{\frac{12}{5}} }^{5}+C(\delta).
\]
Finally, by the Gagliardo-Nirenberg inequality,
\begin{equation}\label{U125}
\bke{\int_{\Omega}|u|^{\frac{12}{5}} }^{5}\le C(\Omega) \|u\|_{L^{2}(\Omega)}^{10}\|\nabla u\|_{L^{2}(\Omega)}^{2},
\end{equation}
Lemma~\ref{UBDD} implies that the time average of $\left(\int_\Omega |u|^{12/5}\right)^5$ over intervals of length $\tau$ is uniformly bounded. The desired results for the first case is obtained by applying Lemma~\ref{EXPTINT} to the above differential inequality for $\mathcal{G}$. 

%the desired result for the first case follows from Lemma~\ref{EXPTINT} and Lemma~\ref{UBDD}.

%% fluid-free case
\smallskip
$\bullet$ \emph{Case 2: $\gamma>0$, $s_{1}\ge 0$, and fluid-free.} 
Because $u \equiv 0$, \eqref{MAR66} simplifies to 
\begin{equation}\label{MAR66_FF}
\begin{aligned}
 &\frac{d}{dt}\int_{\Omega}  n\ln n \varphi^{4}- \frac{d}{dt}\int_{\Omega}n\varphi^{4} 
 + \int_{\Omega} \frac{|\nabla n|^{2}}{n}\varphi^{4} \\
&\quad  =  -\int_{\Omega} \ln n\nabla n\cdot \nabla \varphi^{4} 
  + \int_{\Omega}\nabla n  \cdot ( S(x,n,c)\cdot\nabla c) \varphi^{4}
  + \int_{\Omega}n\ln n (S(x,n,c)\cdot\nabla c  ) \cdot\nabla \varphi^4 \\
&\quad =: II_1 + II_2 + II_3.
\end{aligned}
\end{equation}
The estimate for $II_1$ is the same as \eqref{FORTHTERM}. 

By \eqref{KS02} and Young's inequality, for any $\tilde{\varepsilon}>0$, there exists  $\tilde{r}>1$, and $C=C(\tilde{\varepsilon}, s_{0},\tilde{r})>0$ satisfying 
\begin{align}\label{SSPLIT2}
\begin{aligned}
|S(x,n,c)|&\le \frac{s_{0}}{(s_{1}+c)^{\gamma}}
\\&\le  \tilde{\varepsilon}+\frac{C}{(s_{1}+c)^{\tilde{r}}}.
\end{aligned}
\end{align}
Note also that 
 by \eqref{NABLACL2},
$\int_{\Omega}c=\int_{\Omega}n_{0}$, and \eqref{CLOWER}, there exists  $C=C(\tilde{r}, \Omega, s_1, \int_{\Omega}n_0)>0$   such that
\begin{equation}\label{FEB10_FF}
\sup_{t<T_{\rm max}} \norm{\frac{\nabla c}{(s_{1}+c)^{\frac{3-\tilde{r}}{2}}}(\cdot,t)}_{L^{2}(\Omega)}^{2} \le C.
\end{equation}
Thus, by Proposition~\ref{KEYPRO}, there exists $\tilde{\delta_{1}}>0$ independent of $q\in\overline{\Omega}$ such that for $\delta<\tilde{\delta_{1}}$
\[
C_{2}(\tilde{r})\norm{\frac{\nabla c}{(s_{1}+c)^{\tilde{r}}}}_{L^{2}(\Omega\cap B_{\delta})} \norm{\frac{\nabla c}{(s_{1}+c)^{\frac{3-\tilde{r}}{2}}}}_{L^{2}(\Omega\cap B_{\delta})}^{2}  \le \frac{1}{2}
\]
with $C_2 (\tilde{r}) >0$ as in \eqref{LEMLOCINT2}.
Let $\delta\in (0, \tilde{\delta_{1}})$. From \eqref{LEMLOCINT2}, we have
\begin{align}\label{LEMLOCINT13}
\begin{aligned} 
 \int_{\Omega} &\frac{|\nabla c|^{4}}{(s_{1}+c)^{4\tilde{r}}}\varphi^{4}
   \le
   2C_{2}(\tilde{r})\norm{\frac{\nabla c}{(s_{1}+c)^{\tilde{r}}}}_{L^{2}(\Omega\cap B_{\delta})}
\bke{   \int_{\Omega} \frac{|\nabla n|^{2}}{n}\varphi^{4}   +C(\delta)}
+2C_{2}(\tilde{r})\bke{\int_{\Omega}|u|^{\frac{12}{5}} }^{5}.
\end{aligned}
\end{align}
for a positive constant $C(\delta)=\| \varphi^{2}  \|_{W^{1,\infty}(\R^{2})}^{2}  +\|\varphi^{\frac{4}{3}}\|_{W^{2,2}(\R^{2})}^{3}$ by Lemma~\ref{LEMSTEST}. 

%We apply

Then, we use H\"older's inequality, \eqref{SSPLIT2}, and Young's inequality to obtain 
\begin{align}\label{DEC20232}
\begin{aligned}
II_2
&\le \bke{\int_{\Omega} \frac{|\nabla n|^{2}}{n}\varphi^{4}}^{\frac{1}{2}}
\bke{\int_{\Omega}n^{2}\varphi^{4}}^{\frac{1}{4}}\bkt{\tilde{\varepsilon}\bke{\int_{\Omega} |\nabla c|^{4} \varphi^{4}}^{\frac{1}{4}}+C\bke{\int_{\Omega}\frac{|\nabla c|^{4}}{(s_{1}+c)^{4\tilde{r}}}\varphi^{4}}^{\frac{1}{4}}}
\\&\le\frac{1}{16} \int_{\Omega} \frac{|\nabla n|^{2}}{n}\varphi^{4} +
\frac{1}{16C_{1}}\int_{\Omega}n^{2}\varphi^{4} +  C_{6}\tilde{\varepsilon}^{4}\int_{\Omega}|\nabla c|^{4}\varphi^{4}+C_{7}\int_{\Omega}\frac{|\nabla c|^{4}}{(s_{1}+c)^{4\tilde{r}}}\varphi^{4},
\end{aligned}
\end{align}
for positive constants $C_{6}$, independent of $\tilde{\varepsilon}$, and $C_{7}(\tilde{\varepsilon})$. 

Now let us estimate $\int_{\Omega}|\nabla c|^{4}\varphi^{4}$. 
Because $(a_{1}+a_{2})^{4}\le 8(a_{1}^{4}+a_{2}^{4})$ for $a_{1},a_{2}\ge 0$, we have
\begin{align}\label{DEC20_0}
\begin{aligned}
\int_{\Omega}|\nabla c|^{4}\varphi^{4} &=\|\nabla( c \varphi)- c\nabla \varphi   \|^{4}_{L^{4}(\Omega)}
\\& \le \bke{ \| c\varphi \|_{W^{1,4}(\Omega)}      +\|c\nabla \varphi  \|_{L^{4}(\Omega)}   }^{4}
\\& \le 8\bke{ \| c\varphi \|_{W^{1,4}(\Omega)}^{4}      +\|c\nabla \varphi  \|_{L^{4}(\Omega)}^{4}   }.
\end{aligned}
\end{align}
By H\"older's inequality, the elliptic regularity theory and $W^{2,2}(\Omega)\hookrightarrow W^{1,12}(\Omega)$, there is $C=C(\Omega)>0$ such that
\begin{align}\label{DEC20_1}
\begin{aligned}
\|c\nabla \varphi\|_{L^{4}(\Omega)}^{4} &\le \|c\|_{L^{6}(\Omega)}^{4}\|\nabla \varphi\|_{L^{12}(\Omega)}^{4}
\\& \le C\|n_{0}\|_{L^{1}(\Omega)}^{4}\|\varphi\|_{W^{2,2}(\R^{2})}^{4}.
\end{aligned}
\end{align}
Similarly, the elliptic regularity theory and  $W^{2,\frac{4}{3}}(\Omega) \hookrightarrow W^{1,4}(\Omega)$ yields that with some $C=C(\Omega)>0$
\begin{equation}\label{DEC20_2}
\| c\varphi \|_{W^{1,4}(\Omega)}^{4} \le C\bke{  \| \Delta (c\varphi) \|_{L^{\frac{4}{3}}(\Omega)}^{4}+ \|  c\varphi \|_{L^{\frac{4}{3}}(\Omega)}^{4}  + \|c\nabla \varphi\cdot\nu\|_{W^{\frac{1}{4},\frac{4}{3}}(\partial\Omega)}^{4} },
\end{equation}
where we used
$
\nabla (c\varphi)\cdot\nu =c\nabla \varphi\cdot\nu$ on  $\partial\Omega$.
Using $ \varphi \le 1$, the standard elliptic regularity theory and $\int_{\Omega}n=\int_{\Omega}n_{0}$, we can find
$C=C(\Omega)>0$  satisfying
\begin{align}\label{DEC20_3}
\begin{aligned}
\|  c\varphi \|_{L^{\frac{4}{3}}(\Omega)}^{4}& \le \|c\|_{L^{\frac{4}{3}}(\Omega)}^{4}
 \\&\le C\|n_{0}\|_{L^{1}(\Omega)}^{4}.
\end{aligned}
\end{align}
Note   from the trace inequality that with some
$C=C(\Omega)>0$ 
\begin{equation}\label{DEC20_4}
\|c\nabla \varphi\cdot\nu\|_{W^{\frac{1}{4},\frac{4}{3}}(\partial\Omega)}^{4}  \le C\|c\nabla \varphi\|_{W^{1,\frac{4}{3}}(\Omega)}^{4}.
\end{equation}
By H\"older's inequality and $W^{2,2}(\Omega)\hookrightarrow W^{1,12}(\Omega)$,  note also that with some $C=C(\Omega)>0$ 
\[
\|c\nabla \varphi\|_{W^{1,\frac{4}{3}}(\Omega)}^{4}\le C\bke{ \| c\|_{\frac{3}{2}(\Omega)}^{4}+\|\nabla c\|_{L^{\frac{3}{2}}(\Omega)}^{4} +\|c\|_{L^{4}(\Omega)}^{4}}\|\varphi\|_{W^{2,2}(\R^{2})}^{4}.
\]
Thus, it follows by the standard elliptic regularity theory and $\int_{\Omega}n=\int_{\Omega}n_{0}$ that with some $C=C(\Omega)>0$,
\begin{equation}\label{DEC20_5}
\|c\nabla \varphi\|_{W^{1,\frac{4}{3}}(\Omega)}^{4}\le C\|n_{0}\|_{L^{1}(\Omega)}^{4}\|\varphi\|_{W^{2,2}(\R^{2})}^{4}.
\end{equation}
Similarly, using $(a_{1}+a_{2}+a_{3}+a_{4})^{4}\le 64(a_{1}^{4}+a_{2}^{4}+a_{3}^{4}+a_{4}^{4})$ for $a_{1},a_{2},a_{3},a_{4}\ge0$, H\"older's inequality, the standard elliptic regularity theory, $\int_{\Omega}n=\int_{\Omega}n_{0}$, and $W^{2,2}(\Omega)\hookrightarrow W^{1,12}(\Omega)$, we can find $C=C(\Omega)>0$ satisfying
\begin{align*}
\begin{aligned}
\| \Delta (c\varphi) \|_{L^{\frac{4}{3}}(\Omega)}^{4} & \le 64\bke{\| (c-\Delta c) \varphi  \|_{L^{\frac{4}{3}}(\Omega)}^{4}  +\|  c \varphi  \|_{L^{\frac{4}{3}}(\Omega)}^{4} +\| \nabla  c \cdot \nabla \varphi  \|_{L^{\frac{4}{3}}(\Omega)}^{4}+\|  c \Delta \varphi  \|_{L^{\frac{4}{3}}(\Omega)}^{4} }
\\& \le 64\left( \| n\varphi  \|_{L^{\frac{4}{3}}(\Omega)}^{4} +\|  (1-\Delta)^{-1}{n} \|_{L^{\frac{3}{2}}(\Omega)}^{4} \|\varphi  \|_{L^{12}(\R^{2})}^{4} +\| \nabla(1-\Delta)^{-1}{n}\|_{L^{\frac{3}{2}}(\Omega)}^{4}  \|\nabla \varphi  \|_{L^{12}(\R^{2})}^{4} \right.
\\&\quad \left. +\|  (1-\Delta)^{-1}{n}\|_{L^{4}(\Omega)}^{4} \| \varphi  \|_{W^{2,2}(\R^{2})}^{4} \right)
\\& \le 64\|n_{0}\|_{L^{1}(\Omega)}^{2} \int_{\Omega}n^{2}\varphi^{4}  +C   \|n_{0}\|_{L^{1}(\Omega)}           \| \varphi  \|_{W^{2,2}(\R^{2})}.
\end{aligned}
\end{align*}
Combining it with  \eqref{DEC20_0}--\eqref{DEC20_5} yields that with some $C_{8}=C_{8}(\Omega)>0$ independent of $\tilde{\varepsilon}$ and $q\in\overline{\Omega}$,
\begin{equation}\label{DEC26103}
\int_{\Omega}|\nabla c|^{4}\varphi^{4} \le C_{8}\int_{\Omega}n^{2}\varphi^{4} +C(\delta).
\end{equation}
Now, we pick a sufficiently small $\tilde{\varepsilon}>0$ such that 
\begin{equation}\label{DEC26104}
C_{1}C_{6} C_{8}\tilde{\varepsilon}^{4}\le \frac{1}{32}.
\end{equation}
With such $\tilde{\varepsilon}$, due to Proposition~\ref{KEYPRO}, we can find $\delta_{*} \in (0, \tilde{\delta}_1)$ independent of $q\in\overline{\Omega}$ such that
\begin{equation}\label{DEC26105}
(C_{2}(\tilde{r}) C_{7}+\frac{1}{4}C_{2}(\tilde{r}))\sup_{t<T_{\rm max}}\norm{\frac{\nabla c}{(s_{1}+c)^{\tilde{r}}}}_{L^{2}(\Omega\cap B_{\delta})}\le \frac{1}{16}\quad \mbox{ for all } \delta \in (0, \delta_{*}).
\end{equation}
for $C_2 = C_2 (\tilde{r}) >0$ as in \eqref{LEMLOCINT2}.
Fix $\delta \in (0, \delta_{*})$. Then, by Lemma~\ref{LEMLOCINT}, it follows from \eqref{DEC20232} that
\begin{align}\label{DEC23326}
\begin{aligned}
II_2
%\biggr{|}\int_{\Omega} & \nabla n\cdot  ( S(x,n,c) \cdot \nabla c )   \,\varphi^{4}\biggr{|}
&\le \frac{5}{32} \int_{\Omega} \frac{|\nabla n|^{2}}{n}\varphi^{4} +C_{2}(\tilde{r}) C_{7} \norm{\frac{\nabla c}{(s_{1}+c)^{\tilde{r}}}}_{L^{2}(\Omega\cap B_{\delta})}\int_{\Omega} \frac{|\nabla n|^{2}}{n}\varphi^{4}+C(\delta).
\end{aligned}
\end{align}
 %To estimate the rightmost term of  \eqref{MAR66}, we 
 
By \eqref{SSPLIT2}, Young's inequality, $a^{\frac{4}{3}}|\ln a|^{\frac{4}{3}}\le 16e^{-\frac{4}{3}} a^{\frac{3}{2}}+e^{-\frac{4}{3}}$ for $a\ge0$, Lemma~\ref{LEMSTEST}, H\"older's inequality, and $\int_{\Omega}n=\int_{\Omega}n_{0}$, we obtain the following
\begin{align*}
\begin{aligned}
II_3
%\biggr{|}4&\int_{\Omega}n\ln n  ( S(x,n,c)  \cdot \nabla c  )  \cdot\nabla \varphi  \varphi^{3}\biggr{|}
&\le  4\tilde{\varepsilon}\int_{\Omega} n|\ln n| |\nabla c||\nabla \varphi| \varphi^{3}+ 4C\int_{\Omega} n|\ln n| \frac{|\nabla c|}{(s_{1}+c)^{\tilde{r}}}|\nabla \varphi| \varphi^{3}
\\&
\le C_{6}\tilde{\varepsilon}^{4}\int_{\Omega} |\nabla c|^{4} \varphi^{4}+\frac{1}{4}\int_{\Omega}\frac{|\nabla c|^{4}}{(s_{1}+c)^{4\tilde{r}}}\varphi^{4}+C_{9}\int_{\Omega} n^{\frac{4}{3}}|\ln n|^{\frac{4}{3}}\varphi^{\frac{8}{3}}|\nabla \varphi|^{\frac{4}{3}}
\\&
\le  C_{6}\tilde{\varepsilon}^{4}\int_{\Omega}|\nabla c|^{4}\varphi^{4}+\frac{1}{4}\int_{\Omega}\frac{|\nabla c|^{4}}{(s_{1}+c)^{4\tilde{r}}}\varphi^{4}+C_{9}K^{\frac{4}{3}}\int_{\Omega} (16e^{-\frac{4}{3}} n^{\frac{3}{2}}+e^{-\frac{4}{3}})\varphi^{\frac{10}{3}} 
\\&
\le C_{6}\tilde{\varepsilon}^{4}\int_{\Omega}|\nabla c|^{4}\varphi^{4}
+\frac{1}{4}\int_{\Omega}\frac{|\nabla c|^{4}}{(s_{1}+c)^{4\tilde{r}}}\varphi^{4}
+16 C_{9}K^{\frac{4}{3}} e^{-\frac{4}{3}}\|n_{0}\|_{L^{1}(\Omega)}^{\frac{1}{2}}\bke{\int_{\Omega}n^{2}\varphi^{4}}^{\frac{1}{2}} 
+ C(\delta),
%+ C_{9}K^{\frac{4}{3}}e^{-\frac{4}{3}} \int_{\R^{2}} \varphi^{\frac{10}{3}},
\end{aligned}
\end{align*}
for a constant $C_{9}= C_{9}(\tilde{\varepsilon})>0$. It then follows, by \eqref{DEC26103}, Lemma~\ref{LEMLOCINT}, \eqref{DEC26104} and Young's inequality, that with some $C(\delta)>0$ 
\begin{align*}
\begin{aligned}
II_3
%\biggr{|}4&\int_{\Omega}n\ln n  ( S(x,n,c)  \cdot \nabla c  )  \cdot\nabla \varphi  \varphi^{3}\biggr{|}
&\le \frac{1}{16}\int_{\Omega} \frac{|\nabla n|^{2}}{n}\varphi^{4}+\frac{1}{4}C_{2}(\tilde{r})  \norm{\frac{\nabla c}{(s_{1}+c)^{\tilde{r}}}}_{L^{2}(\Omega\cap B_{\delta})}\int_{\Omega} \frac{|\nabla n|^{2}}{n}\varphi^{4} 
+C(\delta).
\end{aligned}
\end{align*}
Thus, combining it with \eqref{MAR66}, \eqref{FORTHTERM}, \eqref{DEC23326}, and \eqref{DEC26105} yields  
\begin{align*}
\begin{aligned}
\frac{d}{dt} \int_{\Omega}  n\ln n\varphi^{4}-\frac{d}{dt}\int_{\Omega}n\varphi^{4} +\frac{15}{32}\int_{\Omega} \frac{|\nabla n|^{2}}{n}\varphi^{4}
\le C(\delta).
\end{aligned}
\end{align*}

Again, using \eqref{FunctionG} and \eqref{LEM6PF2}, we derive 
\[
\frac{d}{dt}\mathcal{G}(t)+\mathcal{G}(t) +\frac{11}{32}\int_{\Omega} \frac{|\nabla n|^{2}}{n}\varphi^{4}\le C(\delta).
\]
Then, the desired result is a consequence of Lemma~\ref{EXPTINT} and
$
\int_{\Omega}n\varphi^{4}\le\int_{\Omega}n_{0}$. 
\end{proof}

With the localized estimate in Lemma~\ref{LEMLOCINT} and Lemma~\ref{LEMLOGL}, we now employ a standard finite covering argument to upgrade these local bounds into global-in-space uniform estimates. This step bridges the gap between our local analysis and the global regularity needed for the main theorem.

\begin{cor}\label{COR1}
Assume that either $\gamma>\frac12$ and $s_1>0$, or $\gamma>0$, $s_1\ge0$ and the system is fluid-free.
There exists a constant $C=C (\Omega,s_0,s_1,\gamma,\textstyle\int_\Omega n_0,\|u_0\|_{L^2(\Omega)},\|\nabla\Phi\|_{L^\infty(\Omega)})>0$, with the dependence on $u_0,\Phi$ dropped in the fluid-free case, such that
\[\begin{gathered}
\sup_{t<T_{\rm max}} \int_{\Omega}   (n\ln n (\cdot,t)+e^{-1}) \le C, \\	
\int_{t}^{t+\tau}\int_{\Omega}\frac{|\nabla n|^{2}}{n}  \le C  \quad\mbox{ for all }t\in(0,T_{\rm max}-\tau), \\
\int_{t}^{t+\tau}\int_{\Omega}  \frac{|\nabla c|^{4}}{(1+c)^{4\gamma }}  \le C 
\quad\mbox{ for all }t\in(0,T_{\rm max}-\tau), \\
\int_{t}^{t+\tau}\int_{\Omega}   n^{2} \le C \quad\mbox{ for all }t\in(0,T_{\rm max}-\tau),
\end{gathered}
\]
where $\tau:=\min\{1,\frac{1}{2}T_{\rm max}\}$.
\end{cor}

\begin{proof}
Let $\delta_*>0$ be as in Lemma~\ref{LEMLOGL}. Fix any $\delta\in(0,\delta_*)$.
Since $\overline{\Omega}$ is compact, there is a finite subcovering of every open covering $ \bigcup_{q\in\overline{\Omega}}B_{\frac{\delta}{2}}(q)$. 
%Note that every open covering $ \bigcup_{q\in\overline{\Omega}}B_{\frac{\delta}{2}}(q)$, $\delta>0$, of compact set $\overline{\Omega}$ has a finite subcovering. 
Since $x\ln x+e^{-1}\ge0$ for all $x\ge0$ and $\varphi=1$ in $B_{\frac{\delta}{2}}(q)$, the first and second assertions are a direct consequence of Lemma~\ref{LEMLOGL}. 
 By a similar reason, 
we can obtain the third assertion from \eqref{LEMLOCINT2}, \eqref{NABLACL2}, \eqref{U125}, and Lemma~\ref{UBDD}.
The last assertion comes from   the  interpolation inequality,
\[
\|n\|_{L^{2}(\Omega)}^{2}\le C\|n\|_{L^{1}(\Omega)}\int_{\Omega}\frac{|\nabla n|^{2}}{n}+C \|n\|_{L^{1}(\Omega)}^{2},
\]
$\int_{\Omega}n=\int_{\Omega}n_{0}$, and the second assertion.
\end{proof}

Next, we derive a uniform-in-time $L^2$ estimate based on the space-time $L^2$-bound for $n$ from Corollary~\ref{COR1}.

%Now, we derive more regularity properties.

\begin{lemma}\label{RHOL2}
Assume that either $\gamma>\frac12$ and $s_1>0$, or $\gamma>0$, $s_1\ge0$ and the system is fluid-free.
There exists a constant $C=C (\Omega,s_0,s_1,\gamma,\textstyle\int_\Omega n_0, \|n_0\|_{L^2(\Omega)}, \|u_0\|_{L^2(\Omega)},\|\nabla\Phi\|_{L^\infty(\Omega)})>0$, with the dependence on $u_0,\Phi$ dropped in the fluid-free case, such that
\[
\sup_{t<T_{\rm max}} \| n(\cdot,t)\|_{L^{2}(\Omega)}   \le C.
  \]
\end{lemma}

\begin{proof}
Testing $\eqref{KS01}_{1}$ by $n^{p-1}$, $p>1$ yields 
\begin{equation}\label{LPrho}
\frac{1}{p}\frac{d}{dt}\int_{\Omega}n^{p}+(p-1)\int_{\Omega}n^{p-2}|\nabla n|^{2}= (p-1)\int_{\Omega}n^{p-1}\nabla n\cdot S(x,n,c)\cdot \nabla c.
\end{equation}
Taking $p=2$ in \eqref{LPrho} and using \eqref{KS02} and H\"older's inequality, it follows 
\begin{equation}
\frac{1}{2}\frac{d}{dt}\int_{\Omega}n^{2}+\int_{\Omega}|\nabla n|^{2}
\le s_{0}\|n\|_{L^{4}(\Omega)}\|\nabla n\|_{L^{2}(\Omega)}\norm{\frac{\nabla c}{(1+c)^{\gamma}}}_{L^{4}(\Omega)}.
\end{equation}
Using the Gagliardo-Nirenberg inequality in two dimensions,
\[
\|n\|_{L^{4}(\Omega)}\le C \left(\|  n\|_{L^{2}(\Omega)}^{\frac{1}{2}} \|  \nabla n\|_{L^{2}(\Omega)}^{\frac{1}{2}}+\|  n\|_{L^{1}(\Omega)} \right),
\]
and the mass conservation property of $n$, it follows that with some constant $C>0$,
\[
\frac{1}{2}\frac{d}{dt}\int_{\Omega}n^{2}+\int_{\Omega}|\nabla n|^{2}
\le C \left(\|  n\|_{L^{2}(\Omega)}^{\frac{1}{2}} \|  \nabla n\|_{L^{2}(\Omega)}^{\frac{1}{2}}+1 \right)\|\nabla n\|_{L^{2}(\Omega)}\norm{\frac{\nabla c}{(1+c)^{\gamma}}}_{L^{4}(\Omega)}.
\]
Applying Young's inequality to the RHS, we can absorb the $\|  \nabla n\|_{L^{2}(\Omega)}$-terms and obtain an ODE of 
\[
 y'(t)\le Cg(t)y(t)+C, \quad \mbox{ for all } t\in (0, T_{\rm max})
\] 
where we set
\[
y(t):= \|n(\cdot,t)\|_{L^{2}(\Omega)}^{2}+1,\qquad g(t):=\norm{\frac{\nabla c(\cdot,t)}{(1+c(\cdot, t))^{\gamma}}}_{L^{4}(\Omega)}^{4}.
\]

Note from Corollary~\ref{COR1} that there exist positive constants $C_{10}$ and $C_{11}$ satisfying
\[
\int_{t}^{t+\tau}y(s)ds\le C_{10},\qquad \int_{t}^{t+\tau}g(s)ds\le C_{11},\quad\mbox{ for all }t\in(0,T_{\rm max}-\tau),
\]
 where $\tau:=\min\{1,\frac{1}{2}T_{\rm max}\}$. For the case $T_{\rm max}\le2$, we 
apply an ODE-comparison argument to show  
\[
y(t)\le y(0)e^{\int_{0}^{t}Cg(s)ds}+C\int_{0}^{t}e^{\int_{\sigma}^{t}Cg(s)ds}d\sigma, \quad\mbox{ for all }t< T_{\rm max}.
\] 
Then, along with the fact that
\[
\int_{0}^{t}g(s)ds\le \int_{0}^{\tau}g(s)ds+\int_{t-\tau}^{t}g(s)ds\le 2C_{11}, \quad\mbox{ for all }t< T_{\rm max},
\]
we can see that
\begin{equation}\label{MAR21}
y(t)\le \left( y(0)+C \right)e^{CC_{11}}, \quad\mbox{ for all }t<T_{\rm max}.
\end{equation}
Otherwise, if $T_{\rm max}>2$, note that $\tau=1$.
Note also that for any unit interval $[t,t+1]$, $t<T_{\rm max}-1$, we can find $t_{0}\in[t,t+1]$ satisfying
\[
y(t_0)\le C_{10}.
\]
Thus, again by an ODE-comparison argument, we have 
\[
y(t)\le (C_{10}+C)e^{CC_{11}},\quad\mbox{ for all }t<T_{\rm max}.
\]
Combining it with \eqref{MAR21}, we can deduce the desired result.
\end{proof}

Next, we improve the regularity of the fluid velocity.
%More precisely, we derive a uniform-in-time bound for $\|A^\alpha u(\cdot,t)\|_{L^2(\Omega)}$.

\begin{lemma}\label{LEM13}
Assume that $\gamma>\frac12$ and $s_1>0$. Let $\alpha\in(\frac12,1)$ be as in \eqref{KS03}. Then there exists $C=C (\Omega,\alpha,s_0,s_1,\gamma,\int_\Omega n_0,\|n_0\|_{L^2(\Omega)}, \|u_0\|_{L^2(\Omega)},
\|A^\alpha u_0\|_{L^2(\Omega)},\|\nabla\Phi\|_{L^\infty(\Omega)} )>0$ such that
\[
\sup_{t<T_{\rm max}}   \| A^{\alpha} u(\cdot,t)\|_{L^{2}(\Omega)}  \le C.
\]
%where $\alpha\in(\frac{1}{2},1)$ is the number given in \eqref{KS03}.
\end{lemma}

\begin{proof}
First, we show
\begin{equation}\label{MAR210}
\sup_{t<T_{\rm max}}  \| \nabla u(\cdot,t)\|_{L^{2}(\Omega)}  \le C.
\end{equation}
To this end, we apply Helmholtz projection $\mathcal{P}$ to $\eqref{KS01}_{3}$ and
test $Au$. Then,
\[\begin{aligned}
\frac{1}{2}\frac{d}{dt}\int_{\Omega}|\nabla u|^{2}+\int_{\Omega}|Au|^{2}
&\le \int_{\Omega}|\mathcal{P}((u\cdot\nabla) u)||Au|+ \int_{\Omega}|\mathcal{P}(n\nabla \Phi)||Au|
\\&\le \int_{\Omega}|(u\cdot\nabla) u|^{2}+ \int_{\Omega}|n \nabla \Phi|^{2} + \frac{1}{2}\|Au\|^{2}_{L^{2}(\Omega)},
\end{aligned}\]
by the Young's inequality and the projection property $\|\mathcal{P}f\|_{L^{2}(\Omega)}\le \|f\|_{L^{2}(\Omega)}$ for all $f\in L^{2}(\Omega)$.
By the Gagliardo-Nirenberg inequality and elliptic regularity of Stokes operator $\| u\|_{W^{2,2}(\Omega)}\lesssim \|Au \|_{L^{2}(\Omega)}$, there is a constant $C>0$ such that 
\[
\|u\|^{2}_{L^{\infty}(\Omega)} 
%\leq C \|u\|_{L^2(\Omega)}\|u\|_{W^{2,2}(\Omega)} 
\leq C \|Au \|_{L^{2}(\Omega)}\|u\|_{L^2(\Omega)}.
\]
Therefore, we have 
\[
\begin{aligned}
  \int_{\Omega}|(u\cdot\nabla) u|^{2} 
 &\le \|u\|_{L^{\infty}(\Omega)}^{2}\|\nabla u\|_{L^{2}(\Omega)}^{2}  
 \\&\le C\|Au\|_{L^{2}(\Omega)}\|  u\|_{L^{2}(\Omega)}\|\nabla u\|_{L^{2}(\Omega)}^{2}  
 \\ &\le \frac{1}{4}\|Au\|^{2}_{L^{2}(\Omega)} + C \|u\|^{2}_{L^{2}(\Omega)}\|\nabla u\|_{L^{2}(\Omega)}^{4} 
\end{aligned}
\]
and 
\[
 \int_{\Omega}|n \nabla \Phi|^{2} \le  \|\nabla \Phi\|_{L^{\infty}(\Omega)}^{2}\|n\|_{L^{2}(\Omega)}^{2}.
\]
    Combining the above, there is a constant $C>0$ such that 
\[
2\frac{d}{dt}\int_{\Omega}|\nabla u|^{2}+\int_{\Omega}|Au|^{2} \leq 
C \|u\|^{2}_{L^{2}(\Omega)}\|\nabla u\|_{L^{2}(\Omega)}^{4} + 4\|\nabla \Phi\|_{L^{\infty}(\Omega)}^{2}\|n\|_{L^{2}(\Omega)}^{2}.
\]

Using  $u,n\in L^{\infty}(0,T_{\rm max};L^{2}(\Omega))$  
from Lemma~\ref{UBDD} and Lemma~\ref{RHOL2},  and Young's inequality, there exists a constant $C>0$ such that 
\[
h'(t) \le C(\|\nabla u (\cdot, t) \|_{L^{2}(\Omega)}^{2}h(t) +1), \qquad h(t):=\|\nabla u(\cdot,t)\|_{L^{2}(\Omega)}^{2}.
\]
Moreover, by Lemma~\ref{UBDD}, there exists a constant $C_{12}>0$ satisfying
\[
\int_{t}^{t+\tau}\|\nabla u(\cdot,s)\|_{L^{2}(\Omega)}^{2}ds\le C_{12}\quad\mbox{ for all }t\in(0,T_{\rm max}-\tau),
\]
 where $\tau:=\min\{1,\frac{1}{2}T_{\rm max}\}$.  

In the case $T_{\rm max}\le2$,
applying an ODE-comparison argument gives, for all $t< T_{\rm max}$,
\[
h(t)\le h(0)e^{\int_{0}^{t}C\|\nabla u(\cdot,s)\|_{L^{2}(\Omega)}^{2}ds}+C\int_{0}^{t}e^{\int_{\sigma}^{t}C\|\nabla u(\cdot,s)\|_{L^{2}(\Omega)}^{2}ds} d\sigma.
\]
Since
\[
\int_{0}^{t}\|\nabla u(\cdot,s)\|_{L^{2}(\Omega)}^{2}ds\le \int_{0}^{\tau}\|\nabla u(\cdot,s)\|_{L^{2}(\Omega)}^{2}ds+\int_{t-\tau}^{t}\|\nabla u(\cdot,s)\|_{L^{2}(\Omega)}^{2}ds\le 2C_{12}
\]
for all $t<T_{\rm max}$, it follows that
\begin{equation}\label{MAR211}
h(t)\le (h(0)+CC_{7})e^{CC_{12}}\quad\mbox{ for all }t< T_{\rm max}.
\end{equation}
In the case when $T_{\rm max}>2$, note that $\tau=1$. Note also that
for any unit interval $[t,t+1]$,  $t<T_{\rm max}-1$, we can find $t_0\in[t,t+1]$ such that
\[
h(t_0)\le C_{12}.
\]
Thus, again by an ODE-comparison argument, we obtain that 
\[
h(t)\le (C_{12}+C)e^{CC_{12}}\quad\mbox{ for all }t <T_{\rm max}.
\]
Combining it with \eqref{MAR211}, we have \eqref{MAR210}.

Next, we use the smoothing estimates for the Stokes semigroup  and H\"older's inequality
to obtain $\mu>0$ and $C>0$ such that
\[
\begin{aligned}
\|A^{\alpha}u(\cdot,t)\|_{L^{2}(\Omega)}
&\le \|A^{\alpha}u_{0}\|_{L^{2}(\Omega)}+C\int_{0}^{t}(t-s)^{\alpha}e^{-\mu t} \|  u(\cdot,s)\|_{L^{\infty}(\Omega)} \|\nabla u(\cdot,s)\|_{L^{2}(\Omega)} ds
\\&\quad+C\int_{0}^{t}(t-s)^{\alpha}e^{-\mu t}\|\nabla \Phi\|_{L^{\infty}(\Omega)}\|n(\cdot,s)\|_{L^{2}(\Omega)}ds.
\end{aligned}
\]
Choose $p\in(2, \frac{1}{1-\alpha}]$ for $\alpha \in (\frac{1}{2}, 1)$. 
By the interpolation inequality and the embedding theorem $D(A^{\alpha})\hookrightarrow W^{1,p}(\Omega)$, there exists a constant $C>0$ satisfying
\begin{equation}\label{UEMB}
\| u\|_{L^{\infty}(\Omega)}\le C\| u\|_{L^{2}(\Omega)}^{1-\theta} \|A^{\alpha}  u\|_{L^{2}(\Omega)}^{\theta},\qquad \theta=\frac{p}{2(p-1)}\in(0,1).
\end{equation}
It then follows by Lemma~\ref{UBDD},  Lemma~\ref{RHOL2} and \eqref{MAR210} that with some $C>0$
\[
\sup_{t<T_{\rm max}}\|A^{\alpha}u(\cdot,t)\|_{L^{2}(\Omega)}\le C+C \left(\sup_{t<T_{\rm max}}\|A^{\alpha}u(\cdot,t)\|_{L^{2}(\Omega)} \right)^{\theta}.
\]
Therefore, due to Young's inequality, we have the desired result.
\end{proof}

Combining the uniform $L^2$-bound for $u$ in Lemma~\ref{UBDD} with the uniform $D(A^\alpha)$-bound in Lemma~\ref{LEM13}
and the interpolation estimate \eqref{UEMB}, we obtain a uniform-in-time $L^\infty$ bound for $u$.

%In view of Lemma~\ref{UBDD}, Lemma~\ref{LEM13} and \eqref{UEMB}, $u$ is bounded uniformly in time. 

\begin{cor}\label{UINFCOR}
Let $\alpha\in(\frac12,1)$ be as in \eqref{KS03}.
Then there exists a constant \\
$C=C (\Omega,\alpha, s_0, s_1, \gamma, \int_\Omega n_0,\|n_0\|_{L^2(\Omega)}, \|u_0\|_{L^2(\Omega)},\|A^\alpha u_0\|_{L^2(\Omega)},\|\nabla\Phi\|_{L^\infty(\Omega)})>0$ such that
\[
\sup_{t<T_{\rm max}}   \|  u(\cdot,t)\|_{L^{\infty}(\Omega)}  \le C.
\]
\end{cor}

With the uniform bound for $u$ in $L^\infty(\Omega)$ in Corollary~\ref{UINFCOR} and the uniform $L^2$-bound for $n$ in Lemma~\ref{RHOL2},
we next derive a uniform $W^{2,2}$-estimate for $c$ from the elliptic equation \eqref{KS01}$_2$.

\begin{lemma}\label{CBDD}
There exists $C=C (\Omega,\alpha, s_0, s_1, \gamma, \int_\Omega n_0,\|n_0\|_{L^2(\Omega)}, \|u_0\|_{L^2(\Omega)},\|A^\alpha u_0\|_{L^2(\Omega)},\|\nabla\Phi\|_{L^\infty(\Omega)})>0$, with the dependence on $\alpha, u_0,\Phi$ dropped in the fluid-free case, such that
\[
\sup_{t<T_{\rm max}}\|c\|_{W^{2,2}(\Omega)}\le C.
\]
\end{lemma}

\begin{proof}
Applying the standard elliptic regularity theory to $\eqref{KS01}_{2}$, there exists a constant $C>0$ such that
\[
\|c\|_{W^{2,2}(\Omega)}\le C( \|u\cdot \nabla c\|_{L^{2}(\Omega)}+ \|n\|_{L^{2}(\Omega)}).
\]
By H\"older's inequality and the Gagliardo-Nirenberg inequality, there exists a constant $C>0$ satisfying
\begin{align*}
\begin{aligned}
\|u\cdot \nabla c\|_{L^{2}(\Omega)}&\le \|u \|_{L^{\infty}(\Omega)} \|  \nabla c\|_{L^{2}(\Omega)}
\\& \le C\|u \|_{L^{\infty}(\Omega)} \|    c\|_{L^{1}(\Omega)}^{\frac{1}{3}} \|    c\|_{W^{2,2}(\Omega)}^{\frac{2}{3}}.
\end{aligned}
\end{align*}
Thus, using Corollary~\ref{UINFCOR}, $ \int_{\Omega} c = \int_{\Omega} n_0$, 
%$\|c\|_{L^{1}(\Omega)}=\|n_{0}\|_{L^{1}(\Omega)}$, 
Lemma~\ref{RHOL2}, and Young's inequality, we have the desired bound.
\end{proof}

As an immediate consequence of Lemma~\ref{CBDD}, we obtain a uniform-in-time $L^\infty$ bound for $c$.

\begin{cor}\label{CBDDCOR}
There exists $C=C (\Omega,\alpha, s_0, s_1, \gamma, \int_\Omega n_0,\|n_0\|_{L^2(\Omega)}, \|u_0\|_{L^2(\Omega)},\|A^\alpha u_0\|_{L^2(\Omega)},\|\nabla\Phi\|_{L^\infty(\Omega)})>0$, with the dependence on $\alpha, u_0,\Phi$ dropped in the fluid-free case, such that
\[
\sup_{t<T_{\rm max}}   \|  c(\cdot,t)\|_{L^{\infty}(\Omega)}  \le C.
\]
\end{cor}

\begin{proof}
A direct consequence of Lemma~\ref{CBDD} and the Sobolev embedding $W^{2,2}(\Omega) \hookrightarrow L^{\infty}(\Omega)$.
\end{proof}

Next, we bootstrap the integrability of $n$ further.
Using the $p=3$ energy estimate for \eqref{KS01}$_1$, the uniform $W^{2,2}$-bound for $c$, and Lemma~\ref{LNIMPR}, we obtain a uniform-in-time $L^3$ bound for $n$.

\begin{lemma}\label{LEM14}
Assume that either $\gamma>\frac12$ and $s_1>0$, or $\gamma>0$, $s_1\ge0$ and the system is fluid-free.
Then there exists $C=C (\Omega,\alpha, s_0, s_1, \gamma, \int_\Omega n_0,\|n_0\|_{L^2(\Omega)}, \|n_0\|_{L^3(\Omega)}, \|u_0\|_{L^2(\Omega)},\|A^\alpha u_0\|_{L^2(\Omega)},\|\nabla\Phi\|_{L^\infty(\Omega)})>0$, with the dependence on $\alpha, u_0,\Phi$ dropped in the fluid-free case, such that
\[
\sup_{t<T_{\rm max}} \| n(\cdot,t)\|_{L^{3}(\Omega)}   \le C.
  \]
\end{lemma}

\begin{proof}
From \eqref{LPrho} with $p=3$, we have
\[
\frac{1}{3}\frac{d}{dt}\int_{\Omega}n^{3}+\frac{8}{9}\int_{\Omega} |\nabla n^{\frac{3}{2}} |^{2}=\frac{4}{3}\int_{\Omega}n^{\frac{3}{2}}\nabla n^{\frac{3}{2}}\cdot  S(x,n,c)\cdot \nabla c.
\]
Using  H\"older's inequality and \eqref{KS02}, we estimate the right-hand-side as
\begin{align*}
\begin{aligned}
\biggr{|}\frac{4}{3}\int_{\Omega}n^{\frac{3}{2}}\nabla n^{\frac{3}{2}} \cdot & S(x,n,c)\cdot \nabla c \biggr{|}
\le \frac{4}{3}s_{0}\|n\|_{L^{4}(\Omega)}^{\frac{3}{2}}\|   \nabla n^{\frac{3}{2}} \|_{L^{2}(\Omega)}\|c\|_{W^{1,8}(\Omega)}.
\end{aligned}
\end{align*}
Note that using the Gagliardo-Nirenberg inequality of type
\[
\|f\|_{W^{1,8}(\Omega)} \le C(\|f\|_{L^{\infty}(\Omega)}^{\frac{1}{2}}\|f\|_{W^{2,4}(\Omega)}^{\frac{1}{2}}+\|f\|_{L^{\infty}(\Omega)})\quad\mbox{ for all }\,\, f\in C^{2}(\overline{\Omega}),
\]
Corollary~\ref {CBDDCOR}, and the standard elliptic regularity theory,  we can find $C>0$ satisfying
\[
\|c\|_{W^{1,8}(\Omega)}\le C (\|n \|_{L^{4}(\Omega)}^{\frac{1}{2}}+\|u\cdot \nabla c\|_{L^{4}(\Omega)}^{\frac{1}{2}} + 1 ).
\]
This along with the uniform-in-time boundedness of $\|u\|_{L^{\infty}(\Omega)}$ in Corollary~\ref{UINFCOR} and Young's inequality shows that with some $C>0$
\[
\|c\|_{W^{1,8}(\Omega)}\le C (\|n \|_{L^{4}(\Omega)}^{\frac{1}{2}}+1).
\]
Therefore, we have $C>0$ such that 
\[
\frac{1}{3}\frac{d}{dt}\int_{\Omega}n^{3}+\frac{8}{9}\int_{\Omega} |\nabla n^{\frac{3}{2}} |^{2} \leq C \left( \|n\|_{L^{4}(\Omega)}^{2} + 1\right)  \|\nabla n^{\frac{3}{2}} \|_{L^{2}(\Omega)}.
\]
Note also that Lemma~\ref{LNIMPR} with $(f,p,\varepsilon)=(n,3,1)$ and Corollary~\ref{COR1} yields with some $C>0$ independent of $s>1$
\[
\int_{\Omega} n^{4}
 \le  \frac{C}{\ln s} \int_{\Omega} |\nabla n^{\frac{3}{2}}|^{2} + (4C)^{\frac{3}{2}}\bke{\int_{\Omega}n_{0}}^{4}+6s^{4}|\Omega|.
\]
Thus, combining the above estimates, we can find $C >0$  independent of $s>1$ satisfying
\begin{align*}
\begin{aligned}
\frac{1}{3}&\frac{d}{dt}\int_{\Omega}n^{3}+\frac{8}{9}\int_{\Omega} |\nabla n^{\frac{3}{2}} |^{2}
\le C \bke{ \frac{1}{\sqrt{\ln s}}\|   \nabla n^{\frac{3}{2}} \|_{L^{2}(\Omega)}+1+s^{2} }\|   \nabla n^{\frac{3}{2}} \|_{L^{2}(\Omega)}.
\end{aligned}
\end{align*}
We take sufficiently large $s$ and use Young's inequality to obtain $C>0$ such that
\[
 \frac{d}{dt}\int_{\Omega}n^{3}+ \int_{\Omega} |\nabla n^{\frac{3}{2}} |^{2}\le C.
\]
This implies, by the Gagliardo-Nirenberg type inequality,
\[
\|f\|_{L^{3}(\Omega)}^{3}\le \|\nabla f^{\frac{3}{2}}\|_{L^{2}(\Omega)}^{2}+C\|f\|_{L^{1}(\Omega)}^{3}\quad\mbox{for all}\quad f\in C^{1}(\overline{\Omega}),
\] 
and the mass conservation property of $n$, that with some $C>0$,
\[
 \frac{d}{dt}\int_{\Omega}n^{3}+\int_{\Omega}n^{3}\le C.
\]
Therefore, by an ODE-comparison argument, we can conclude the desired result.
\end{proof}

We conclude the proof by showing that the local classical solution from Lemma~\ref{LWLEM} is uniformly bounded in time.
To this end, we first upgrade the regularity of $c$ and then apply a Moser--type iteration. 

\medskip 

\begin{pfthm1}
Applying the elliptic regularity theory to the $c$-equation gives that there exists $C=C(\Omega)>0$ such that
\[
\|c\|_{W^{2,3}(\Omega)}\le C\bke{    \|n\|_{L^{3}(\Omega)}+ \|u\cdot \nabla c\|_{L^{3}(\Omega)}          }.
\]
Note from H\"older's inequality and $W^{2,2}(\Omega)\hookrightarrow W^{1,3}(\Omega)$ that
\[
 \|u\cdot \nabla c\|_{L^{3}(\Omega)}  \le \|u\|_{L^{\infty}(\Omega)}\| c\|_{W^{2,2}(\Omega)}.
\]
Thus, due to Corollary~\ref{UINFCOR}, Lemma~\ref{CBDD} and Lemma~\ref{LEM14}, it follows that with some $C>0$
\[
\sup_{t<T_{\rm max}}\| c(\cdot,t)\|_{W^{2,3}(\Omega)}\le C. 
\]
Since $W^{2,3}(\Omega)\hookrightarrow W^{1,\infty}(\Omega)$, using the uniform-in-time bound of $\|\nabla c\|_{L^{\infty}(\Omega)}$ and Moser-type iteration argument, we have 
\[
\sup_{t<T_{\rm max}}\|n(\cdot,t)\|_{L^{\infty}(\Omega)}\le C.
\]
\end{pfthm1}

\section{Large--Time Behavior: Proof of Theorem~\ref{THM3}}\label{S:Asympt}

In this section, we prove the asymptotic stabilization of the fluid-free system as stated in Theorem~\ref{THM3}. We begin by establishing a general interpolation inequality involving the H\"older norm. This inequality will be used later in upgrading exponential convergence in $L^2(\Omega)$ to  exponential convergence in $L^{\infty}(\Omega)$.
\begin{lemma}\label{EHRLEM}
Let $\Omega \subset \mathbb{R}^d$, $d\ge1$, be a bounded Lipschitz domain, $\theta \in (0,1)$, and $p \geq 1$. Then there exists $C = C(\Omega, d, \theta, p) > 0$ such that 
\begin{equation}\label{ineq:main}
  \norm{f}_{L^\infty(\Omega)}
  \leq C\left([f]_{C^\theta(\overline\Omega)}^{\,\frac{d}{p\theta + d}}\,\norm{f}_{L^p(\Omega)}^{\frac{p\theta}{p\theta + d}}
  + \norm{f}_{L^p(\Omega)}\right)\quad \mbox{ for all }\quad f \in C^\theta(\overline\Omega).
\end{equation}
In particular, with some $C = C(\Omega, d, \theta, p) > 0$
\begin{equation}\label{ineq:full}
  \norm{f}_{L^\infty(\Omega)}
  \leq C\,\norm{f}_{C^\theta(\overline\Omega)}^{\,\frac{d}{p\theta + d}}\,\norm{f}_{L^p(\Omega)}^{\frac{p\theta}{p\theta + d}}\quad \mbox{ for all }\quad f \in C^\theta(\overline\Omega).
\end{equation}
\end{lemma}

\begin{proof}
First, we show \eqref{ineq:main}. In the case of  $[f]_{C^\theta} = 0$, $f$ is constant on $\overline\Omega$ and thus,
$
  \norm{f}_{L^\infty}
  = \abs{\Omega}^{-\frac{1}{p}}\,\norm{f}_{L^p}$ implies  \eqref{ineq:main}.
It remains to consider   the case $[f]_{C^\theta} > 0$.
Fix $x \in \overline\Omega$ and let $0 < r \leq R_0 := {\rm diam}(\Omega)$. For every $y \in B_r(x) \cap \Omega$,
\[
  \abs{f(x)}
  \leq \abs{f(x) - f(y)} + \abs{f(y)}
  \leq [f]_{C^\theta}\,r^\theta + \abs{f(y)}.
\]
Averaging over $y \in B_r(x) \cap \Omega$ gives
\[
  \abs{f(x)}
  \leq [f]_{C^\theta}\,r^\theta
  + \frac{1}{\abs{B_r(x) \cap \Omega}}\int_{B_r(x) \cap \Omega}\abs{f(y)}\,dy,
\]
and by H\"older's inequality,
\[
  \frac{1}{\abs{B_r(x) \cap \Omega}}\int_{B_r(x) \cap \Omega}\abs{f(y)}\,dy
  \leq \abs{B_r(x) \cap \Omega}^{-1/p}\,\norm{f}_{L^p(\Omega)}.
\]
Note that since $\Omega$ is Lipschitz, an interior cone condition provides $C= C(\Omega) > 0$ satisfying
\[
  \abs{B_r(x) \cap \Omega} \geq \frac{r^d}{C},
  \quad\mbox{ for all } x \in \overline\Omega,\;\;   r \in (0, R_0].
\]
Thus, it follows that with some $C_{13}=C_{13}(\Omega,p)>0$
\begin{equation}\label{eq:pointwise}
  \abs{f(x)} \leq [f]_{C^\theta}\,r^\theta + C_{13}\,r^{-d/p}\,\norm{f}_{L^p(\Omega)}   \quad\mbox{ for all } x \in \overline\Omega,\;\;   r \in (0, R_0].
\end{equation}
Define the right-hand side as
\[
  g(r) := [f]_{C^\theta}\,r^\theta + C_{13}\,r^{-d/p}\,\norm{f}_{L^p},
  \qquad r \in (0, R_0],
\]
and let 
\[
  r_* := \left(\frac{d\,C_{13}\,\norm{f}_{L^p}}{p\,\theta\,[f]_{C^\theta}}\right)^{\!\frac{p}{p\theta+d}}.
\]
We conisder two cases, $r_* \leq R_0$ and $r_* > R_0$, separately.

$\bullet$ \emph{Case~1: $r_* \leq R_0$.}
Note that $g'\ge0$ in $(0,r_{*}]$ and $g'\le0$ in $[r_{*},R_{0}]$.
Plugging $r_{*}$ into the first term of $g$,
\[
\begin{aligned}
  [f]_{C^\theta}\, r_*^\theta
  &= [f]_{C^\theta}\cdot\left(\frac{d\,C_{13}}{p\,\theta}\right)^{\!\frac{p\theta}{p\theta + d}}
  \left(\frac{\norm{f}_{L^p}}{[f]_{C^\theta}}\right)^{\!\frac{p\theta}{p\theta + d}}
  \\&= \left(\frac{d\,C_{13}}{p\,\theta}\right)^{\!\frac{p\theta}{p\theta + d}}
  [f]_{C^\theta}^{\,\frac{d}{p\theta + d}}\,\norm{f}_{L^p}^{\,\frac{p\theta}{p\theta + d}}.
\end{aligned} 
\]
A  similar computation for the second term of $g$ gives the same structure. Therefore, along with \eqref{eq:pointwise} we can deduce that there exists $C = C(\Omega, d, \theta, p) > 0$ satisfying
\begin{equation}\label{eq:case1}
  \norm{f}_{L^\infty(\Omega)}\leq C \, [f]_{C^\theta}^{\,\frac{d}{p\theta + d}}\,\norm{f}_{L^p}^{\,\frac{p\theta}{p\theta + d}}.
\end{equation}
Namely, \eqref{ineq:main} holds.

$\bullet$ \emph{Case~2: $r^* > R_0$.}
Note that $g'\ge0$ in $(0,R_{0}]$. Note also that
the condition $r^* > R_0$ implies
\begin{equation}\label{eq:case2cond}
  [f]_{C^\theta}
  < \frac{d\,C_{13}}{p\,\theta}\,R_0^{-(p\theta+d)/p}\,\norm{f}_{L^p}.
\end{equation}
Substituting $r = R_0$ into $g$ gives
\[
  g(R_{0})
  \leq [f]_{C^\theta}\,R_0^\theta + C_{13}\,R_0^{-d/p}\,\norm{f}_{L^p}.
\]
We use \eqref{eq:case2cond} to estimate the first term as
\[
  [f]_{C^\theta}\,R_0^\theta
  < \frac{d\,C_{13}}{p\,\theta}\,R_0^{\theta-(p\theta+d)/p}\,\norm{f}_{L^p}
  = \frac{d\,C_{13}}{p\,\theta}\,R_0^{-d/p}\,\norm{f}_{L^p}.
\]
Therefore,   along with \eqref{eq:pointwise} we can deduce that
\begin{equation}\label{eq:case2}
  \norm{f}_{L^\infty(\Omega)}
  \leq \left(\frac{d\,C_{13}}{p\,\theta} + C_{13}\right)R_0^{-d/p}\,\norm{f}_{L^p}.
\end{equation}
Namely, \eqref{ineq:main} holds.

Next, we show \eqref{ineq:full}. Note that
\[
  \norm{f}_{L^p(\Omega)} \leq \abs{\Omega}^{1/p}\,\norm{f}_{L^\infty(\Omega)}
  \leq \abs{\Omega}^{1/p}\,\norm{f}_{C^\theta(\overline\Omega)}
\]
and thus, 
\[
  \norm{f}_{L^p}
  = \norm{f}_{L^p}^{\,\frac{d}{p\theta + d}}\cdot\norm{f}_{L^p}^{\frac{p\theta}{p\theta + d}}
  \leq \left(\abs{\Omega}^{1/p}\,\norm{f}_{C^\theta}\right)^{\frac{d}{p\theta + d}}\cdot\norm{f}_{L^p}^{\frac{p\theta}{p\theta + d}}.
\]
Substituting it into the rightmost term of \eqref{ineq:main} yields \eqref{ineq:full}.
\end{proof}
Now we establish 
the uniform-in-time boundedness of solutions in the H\"older norm.
\begin{lemma}\label{HOLDERNORM}
Let $(n,c)$ be a unique bounded classical solution obtained in Theorem~\ref{THM1}. Then for any $\theta\in(0,1)$ there exists $C>0$ satisfying
\[
\|n\|_{C^{\theta,\theta/2}(\overline{\Omega}\times[t,t+1])}\le C\quad\mbox{ for all }t>1.
\]
\end{lemma}
\begin{proof}
The $n$-equation can be rewritten as 
$n_{t}=\Delta n-\nabla \cdot \vec{a}_{1}(x,t)$, where $\vec{a}_{1}(x,t)=n\,S(x,n,c)\nabla c$. Since Theorem~\ref{THM1} along with \eqref{KS02} yields $\vec{a}_{1}$ is bounded in $L^{\infty}(0,\infty;L^{\infty}(\Omega))$ we can deduce the desired result by standard H\"older regularity, for instance \cite[Thm.~1.3]{Porzio_1993}.
\end{proof}
We recall the following two lemmas from \cite[Lem.~3 and Lem.~4]{Ahn_Kang_Lee_2019}.
\begin{lemma}\label{cruciallem}
Let $\Omega\subset \R^{d}$, $d\ge1$, be a smoothly bounded domain. Suppose that $\mathcal{H}\in C^{1}( (0,\infty)  )$ is positive such that $\mathcal{H}^{\prime}>0$, and $f\in C^{2}(\overline{\Omega})$ is positive such that $\nabla f\cdot n=0$ on $\partial\Omega$. Then, for $k\in\R$,  $\delta>0$, and  $\Theta(f):=\int_{1}^{f}\frac{1}{\mathcal{H}(\tau)}d\tau$,  we have
\begin{equation}\label{cruciallem_1}
\int_{\Omega}\frac{\mathcal{H}^{\prime}(f)}{\mathcal{H}^{k}(f)}\abs{\nabla f}^{4} \le
\left\{
\begin{array}{ll}
\displaystyle\bke{    \frac{2+\sqrt{d}}{4-k}     }^{2}\int_{\Omega}\frac{\mathcal{H}^{4-k}(f)}{\mathcal{H}^{\prime}(f)}\abs{D^{2}\Theta(f)     }^{2},\quad\quad&  k\neq4,\\
\displaystyle\bke{    2+\sqrt{d}     }^{2}\int_{\Omega}\frac{\abs{\log \mathcal{H}(f)}^{2}}{\mathcal{H}^{\prime}(f)}\abs{D^{2}\Theta(f)     }^{2},\quad\quad&  k=4,
\end{array}
\right.
\end{equation}
\begin{equation}\label{cruciallem_2}
\int_{\Omega}\!\!\frac{\mathcal{H}^{2-k}(f)}{\mathcal{H}^{\prime}(f)}\abs{  D^{2}f    }^{2}\le\frac{\delta+1}{\delta}\!\int_{\Omega}\frac{\mathcal{H}^{\prime}(f)}{\mathcal{H}^{k}(f)}\abs{\nabla f}^{4}+\bke{1+\delta}\!\!\int_{\Omega}\!\!\frac{\mathcal{H}^{4-k}(f)}{\mathcal{H}^{\prime}(f)}\abs{D^{2}\Theta(f)     }^{2}.
\end{equation}
%In particular, we have
%\begin{equation}\label{d2loglemgoal1}
%\int_{\Omega}\frac{\abs{\nabla f}^{4}}{f^{2}}dx\le\bke{ 1+\frac{\sqrt{d}}{2}   }^{2}\int_{\Omega}f^{2}\abs{D^{2}\log f    }^{2}dx,
%\end{equation}
%\begin{equation}\label{d2loglemgoal2}
%\int_{\Omega} \abs{D^{2}f}^{2}dx\le \bke{2+\frac{\sqrt{d}}{2}    }^{2}\int_{\Omega}f^{2}\abs{D^{2}\log f    }^{2}dx.
%\end{equation}
\end{lemma}
\begin{lemma}\label{lem3prop2}
Let $\Omega\subset \R^{d}$, $d\ge1$, be a smoothly bounded domain and $\mu_1$ be the first positive Neumann eigenvalue of the Laplace operator $-\Delta$ on $\Omega$. For $f\in C^{2}(\overline{\Omega})$, such that $\int_{\Omega}f=0$ and $\nabla f\cdot n=0$ on $\partial\Omega$, we have
\[
\int_{\Omega}\abs{ f}^{2}\le \frac{1}{\mu_{1}} \int_{\Omega}\abs{\nabla f}^{2},\quad\qquad \int_{\Omega}\abs{\nabla f}^{2}\le \frac{1}{\mu_{1}} \int_{\Omega}\abs{\Delta f}^{2}.
\]
\end{lemma}

We conclude the proof by establishing the exponential-in-time stabilization of the solution. To this end, we first derive the exponential convergence of $c$ in $H^1(\Omega)$ and employ it to obtain the exponential convergence of $n$ in $L^\infty(\Omega)$ by using the $n$-equation and the interpolation inequality involving the H\"older norm.

\medskip

\begin{pfthm3}
Denote $\overline{n_0}:=\frac{1}{|\Omega|}\int_{\Omega}n_{0}$, $\tilde{n} := n - \overline{n_0}$, and $\tilde{c} := c - \overline{n_0}$. Note that  
\[
\tilde{n}_t - \Delta \tilde{n} = -\nabla \cdot (n S(x,n,c)\cdot \nabla  c ),\qquad -\Delta \tilde{c} + \tilde{c} = \tilde{n} \qquad\mbox{ in } \,\,\Omega\times(0,\infty).
\]
Testing the $\tilde{n}$-equation by $\tilde{c}$ and using integration by parts gives 
\begin{equation}\label{CH1}
\frac{1}{2}\frac{d}{dt}\int_{\Omega}( \tilde{c}^{2}+|\nabla \tilde{c}|^{2})+\int_{\Omega }(|\nabla \tilde{c}|^{2}+|\Delta \tilde{c}|^{2}) = \int_{\Omega} n\nabla \tilde{c}\cdot S \cdot \nabla  c .
\end{equation}
Now we derive estimates for the right-hand side.

$\bullet$ \emph{Case (i): $\hat{S}$ is negative semi-definite.}   Since 
\[
\int_{\Omega} n\nabla \tilde{c}\cdot S \cdot \nabla  c  = \int_{\Omega} n\nabla  c \cdot \hat{S} \cdot \nabla  c \le 0,
\]
it follows by \eqref{CH1} and Lemma~\ref{lem3prop2} that
\[
\frac{1}{2}\frac{d}{dt}\int_{\Omega}( \tilde{c}^{2}+|\nabla \tilde{c}|^{2})+\mu_{1}\int_{\Omega}( \tilde{c}^{2}+|\nabla \tilde{c}|^{2}) \le 0.
\]
This yields, with some $C>0$
\begin{equation}\label{H1_decay}
\|\tilde{c}(\cdot, t) \|_{H^1(\Omega)} \le C  e^{-2\mu_{1}t} \quad \text{for all } t > 0.
\end{equation}
To obtain the convergences in better spaces, we note from Theorem~\ref{THM1} that $\tilde{n}$ is uniformly bounded. This implies, applying the standard elliptic regularity theory to $-\Delta \tilde{c} + \tilde{c} = \tilde{n}$, that $\tilde{c}$ is uniformly bounded in $W^{2,p}(\Omega)$ for any $p \in (1, \infty)$. Choosing $p > 2$, the Gagliardo-Nirenberg interpolation inequality provides $\kappa \in (0, 1)$ and $C > 0$ such that
\[
\|\tilde{c}\|_{W^{1,\infty}(\Omega)} \le C \|\tilde{c}\|_{W^{2,p}(\Omega)}^\kappa \|\tilde{c}\|_{H^1(\Omega)}^{1-\kappa}.
\]
Thus, by \eqref{H1_decay}, with some $C>0$
\begin{equation}\label{CWINFDEC}
\|\tilde{c}(\cdot, t)  \|_{W^{1,\infty}(\Omega)}\le C  e^{-2(1-\kappa)\mu_{1}t} \quad \text{for all } t > 0.
\end{equation}
Now, testing the $\tilde{n}$-equation by $\tilde{n}$ and using \eqref{KS02}, Lemma~\ref{ULE}, and Young's inequality, we observe that with some $C=C(\Omega, s_{0},s_{1},\gamma)>0$
\[
\begin{aligned}
\frac{1}{2}\frac{d}{dt}\int_{\Omega}\tilde{n}^{2}+\int_{\Omega}|\nabla \tilde{n}|^{2}&=\int_{\Omega} n \nabla \tilde{n}\cdot S(x,n,c)\cdot \nabla  \tilde{c}
\\&\le \frac{1}{2}\int_{\Omega}|\nabla \tilde{n}|^{2}+C\int_{\Omega}n^{2}|\nabla \tilde{c}|^{2}.
\end{aligned}
\]
By the Poincar\'e inequality
\[
\int_{\Omega}  f ^{2} \le \frac{1}{\mu_{\rm P}}\int_{\Omega}|\nabla f|^{2}\quad \mbox{ for all }\quad f\in C^{1}(\overline{\Omega}) \mbox{ with } \int_{\Omega}f=0,
\]
$n\in L^{\infty}(0,\infty;L^{\infty}(\Omega))$, and \eqref{CWINFDEC}, it follows that with some $C>0$
\[
\frac{d}{dt}\int_{\Omega}\tilde{n}^{2}+\mu_{\rm P}\int_{\Omega}   \tilde{n} ^{2}\le C e^{-4(1-\kappa)\mu_{1}t}\quad \text{for all } t > 0.
\]
This implies $\tilde{n} \rightarrow 0$ in $L^{2}(\Omega)$ exponentially in time. Thus, we arrive at the exponential decay  $\tilde{n} \rightarrow 0$ in $L^{\infty}(\Omega)$ by Lemma~\ref{HOLDERNORM} and \eqref{ineq:full} with $d=p=2$.

$\bullet$ \emph{Case (ii): $S=\frac{s_{0}}{c^{\gamma}} \mathbb{I}$, $\gamma\ge1$, $s_{0}< s_{\star} $, and $\Omega$ is convex.}  
Replacing $n$ with $  c-\Delta c$ on the right-hand side of \eqref{CH1},
\begin{equation}\label{FEB23}
\begin{aligned}
\int_{\Omega} n\nabla \tilde{c}\cdot S \cdot \nabla  c &= s_0 \int_{\Omega}  c^{-\gamma} n |\nabla c|^2 \\&=  s_0 \int_{\Omega}  c^{1-\gamma} |\nabla c|^2 -s_0\int_{\Omega} c^{-\gamma} \Delta c   |\nabla c|^2 .
\end{aligned}
\end{equation}
To handle the rightmost term, we apply integration by parts and the vector identity $\nabla |\nabla f|^2 = 2 D^2 f \cdot \nabla f$ as
\[
\begin{aligned}
-s_{0} \int_{\Omega} c^{-\gamma} \Delta c |\nabla c|^2 &= s_{0} \int_{\Omega} \nabla c \cdot \nabla(c^{-\gamma} |\nabla c|^2) \\&= s_{0} \int_{\Omega}2   c^{-\gamma} \nabla c \cdot D^2 c \cdot \nabla c  -\gamma  c^{-\gamma-1} |\nabla c|^4. 
\end{aligned}
\]
Setting
\[
\mathcal{H}(c) := c^\gamma,\qquad  
\Theta(c):=\int_{1}^{c}\frac{1}{\mathcal{H}(\tau)}d\tau,
\]
and using the pointwise estimate   
\[
\frac{c^{1+\gamma}}{\gamma} |D^2 \Theta|^2 = \frac{1}{\gamma} c^{1-\gamma} |D^2 c|^2 - 2 c^{-\gamma} \nabla c \cdot D^2 c \cdot \nabla c  + \gamma c^{-\gamma-1} |\nabla c|^4,
\]
we arrive at
\begin{equation}\label{FEB24}
- s_{0} \int_{\Omega} c^{-\gamma} \Delta c |\nabla c|^2=    \frac{s_{0} }{\gamma} \int_{\Omega} c^{1-\gamma} |D^2 c|^2 -   \frac{s_{0} }{\gamma} \int_{\Omega}  c^{1+\gamma}  |D^2 \Theta|^2.
\end{equation}
Note from \eqref{cruciallem_1} with $k=d=2$  that
\[
 \frac{4 \gamma^{2} }{(2+\sqrt{2})^2}\int_{\Omega} c^{-\gamma-1} |\nabla c|^4 \le  \int_{\Omega} c^{1+\gamma} |D^2 \Theta(c)|^2.
\]
Thus, by  \eqref{cruciallem_2} with   $\delta = \frac{2+\sqrt{2}}{2}$, we see that
\[
\frac{2}{9+4\sqrt{2}}\int_{\Omega} c^{1-\gamma} |D^2 c|^2 \le  \int_{\Omega} c^{1+\gamma} |D^2 \Theta(c)|^2.
\]
Substituting this into \eqref{FEB24} gives
\[
\begin{aligned}
-s_0 \int_{\Omega} c^{-\gamma} \Delta c |\nabla c|^2 &\le \frac{s_{0}}{\gamma} \int_{\Omega} c^{1-\gamma} |D^2 c|^2 - \frac{s_{0}}{\gamma} \frac{2}{9+4\sqrt{2}} \int_{\Omega} c^{1-\gamma} |D^2 c|^2 \\
&= \frac{s_{0}}{\gamma} \frac{7+4\sqrt{2}}{9+4\sqrt{2}} \int_{\Omega} c^{1-\gamma} |D^2 c|^2.
\end{aligned}
\]
It follows from \eqref{FEB23} that\[
\int_{\Omega} n\nabla \tilde{c}\cdot S \cdot \nabla c \le s_0 \int_{\Omega} c^{1-\gamma} |\nabla \tilde{c}|^2 + \frac{s_{0}}{\gamma} \frac{7+4\sqrt{2}}{9+4\sqrt{2}} \int_{\Omega} c^{1-\gamma} |D^2 c|^2.
\]
Recall from Lemma~\ref{ULE} that there exists $C_{\Omega}>0$ such that
\[
\sup_{x\in\Omega,t>0} c^{-1}\le \frac{C_{\Omega}}{\int_{\Omega}n_{0}}.
\]
Note that 
since $\Omega$ is convex and $\nabla c\cdot\nu=0$ on $\partial\Omega$, we have $\partial_\nu |\nabla c|^2 \le 0$ on $\partial\Omega$ and thus, 
\[
\int_{\Omega} |D^2 c|^2=\int_{\Omega} |\Delta c|^2+\frac{1}{2}\int_{\Omega}\partial_\nu |\nabla c|^2 \le \int_{\Omega} |\Delta c|^2.
\]  
Combining the above, we can deduce that
\[
\int_{\Omega} n\nabla \tilde{c}\cdot S \cdot \nabla c \le s_0 \bke{\frac{C_{\Omega}}{\int_{\Omega}n_{0}}}^{\gamma-1} \int_{\Omega} |\nabla \tilde{c}|^2 + \frac{s_{0}}{\gamma}\frac{7+4\sqrt{2}}{  9+4\sqrt{2} }\bke{\frac{C_{\Omega}}{\int_{\Omega}n_{0}}}^{\gamma-1} \int_{\Omega} |\Delta \tilde{c}|^2.
\]
Plugging this estimate into \eqref{CH1},  
\[
\frac{1}{2}\frac{d}{dt}\int_{\Omega}( \tilde{c}^{2}+|\nabla \tilde{c}|^{2})  + \bkt{1 - \frac{s_0}{\gamma}\frac{7+4\sqrt{2}}{  9+4\sqrt{2} }\bke{\frac{C_{\Omega}}{\int_{\Omega}n_{0}}}^{\gamma-1}}\int_{\Omega}|\Delta \tilde{c}|^{2} \le \bket{  s_0 \bke{\frac{C_{\Omega}}{\int_{\Omega}n_{0}}}^{\gamma-1}-1}\int_{\Omega }|\nabla \tilde{c}|^{2}.
\]
By Lemma~\ref{lem3prop2}, the right-hand side  can be computed as
\[
\bket{  s_0 \bke{\frac{C_{\Omega}}{\int_{\Omega}n_{0}}}^{\gamma-1}-1}\int_{\Omega }|\nabla \tilde{c}|^{2}\le \frac{1}{\mu_{1}}\bket{  s_0 \bke{\frac{C_{\Omega}}{\int_{\Omega}n_{0}}}^{\gamma-1}-1}_{+}\int_{\Omega }|\Delta \tilde{c}|^{2},
\]
where 
\[
\bket{f}_{+}:=\begin{cases}
0, & \text{if } f\le 0,\\[1mm]
f, & \text{if } f>0.
\end{cases}
\]
Thus,
\[
\frac{1}{2}\frac{d}{dt}\int_{\Omega}( \tilde{c}^{2}+|\nabla \tilde{c}|^{2})  +  \bkt{1 - \frac{s_0}{\gamma}\frac{7+4\sqrt{2}}{  9+4\sqrt{2} }\bke{\frac{C_{\Omega}}{\int_{\Omega}n_{0}}}^{\gamma-1} -\frac{1}{\mu_{1}}\bket{  s_0 \bke{\frac{C_{\Omega}}{\int_{\Omega}n_{0}}}^{\gamma-1}-1}_{+}}\int_{\Omega}|\Delta \tilde{c}|^{2} \le0,
\]
where the coefficient of $\int_{\Omega}|\Delta \tilde{c}|^{2}$ is strictly positive provided that 
\begin{equation}\label{SSTAR}
  s_0< s_{\star}:=\frac{1+ \mu_{1}^{-1}}{\frac{1}{\gamma}\frac{7+4\sqrt{2}}{9+4\sqrt{2}}+\mu_{1}^{-1}}\bke{\frac{\int_{\Omega}n_{0}}{C_{\Omega}}}^{\gamma-1}.
\end{equation}
Therefore, again by Lemma~\ref{lem3prop2},   for some $C>0$
\[
 \frac{d}{dt}\int_{\Omega}( \tilde{c}^{2}+|\nabla \tilde{c}|^{2}) + C \int_{\Omega }( \tilde{c}^{2}+|\nabla \tilde{c}|^{2})\le 0,
\]
which implies the exponential decay of $\|\tilde{c}(\cdot, t)\|_{H^1(\Omega)}$ similar to \eqref{H1_decay}. Now,  as in Case (i), utilizing the uniform boundedness of $\tilde{n}$ from Theorem~\ref{THM1} and standard elliptic regularity, we can upgrade this exponential decay $\tilde{c}\rightarrow 0 $ to $W^{1,\infty}(\Omega)$ topology. Repeating the same argument as in Case (i), 
we can deduce $\tilde{n} \rightarrow 0$ in $L^{2}(\Omega)$ exponentially in time and thus, eventually we have the exponential decay  $\tilde{n} \rightarrow 0$ in $L^{\infty}(\Omega)$ by Lemma~\ref{HOLDERNORM} and \eqref{ineq:full} with $d=p=2$.
\end{pfthm3}

%%%%
\bibliographystyle{amsplain}

\end{document}